\documentclass[11pt]{amsart}
\usepackage{amsfonts}
\usepackage{enumerate}
\usepackage{graphicx}
\def\cal{\mathcal}
\newtheorem{theorem}[subsection]{Theorem}
\newtheorem{proposition}[subsection]{Proposition}

\newtheorem{lemma}[subsection]{Lemma}

\numberwithin{equation}{subsection}
\def \Amalg{\mathbin{\raise .5pt%
	\hbox{$\scriptstyle \amalg$}}}

\title[Markov Partitions and Hyperbolic Components]{Persistent Markov partitions and hyperbolic components of rational maps}
\author{Mary Rees}

\begin{document}
\begin{abstract}
Markov partitions persisting in  a neighbourhood of hyperbolic components of rational maps were constructed under the condition that closures of Fatou components are disjoint in \cite{R1}. Given such a partition, we characterize all nearby hyperbolic components in terms of the symbolic dynamics. This means we can count them, and also obtain topological information. We also determine extra conditions under which all nearby type IV hyperbolic components are given by matings. These are probably the first known results of this type.\end{abstract}

\maketitle

\section{Introduction}\label{1}

Studies in complex dynamics are invariably as much about variation of dynamics in a parameter space as about the dynamics of individual maps in the parameter space. This an attractive area of study because of the rich variation, even in parameter spaces of small dimension: complex dimension one is usually ample. The low dimension  is one reason why the study is also relatively tractable. Another reason is the availability of extremely useful tools, of which the  Yoccoz puzzle is a key example. The Yoccoz puzzle is a sequence of successively finer Markov partitions (obtained by backwards iteration of a single Markov partition), and the use of Markov partitions is of course ubiquitous in dynamics. There is an associated Yoccoz ``parapuzzle'', a sequence of successively finer partitions of parameter space. This illustrates the most basic principle in dynamics, that information about parameter space near a map $f$ can be obtained from the dynamics of $f$. In complex dynamics, this informations is often remarkably complete. There have been a number of generalisations of the Yoccoz puzzle and parapuzzle, stretching back some 20 years. See, for example, Roesch \cite{Roesch}. In \cite{R1} we gave a rather general construction of a Markov partition for a geometrically finite rational map for which all Fatou components are topological discs and  have disjoint closures. We use the following definition.

\noindent {\textbf{Definition}} 
A {\em{Markov partition for $f$}} is a set ${\cal{P}}=\{ P_1,\cdots P_r\} $ such that:
\begin{itemize}
\item $\overline{{\rm{int}}(P_i)}=P_i$;
\item $P_i$ and $P_j$ have disjoint interiors if $i\ne j$;
\item $\cup _{i=1}^rP_i=\overline{\mathbb C}$;
\item each $P_i$ is a union of connected components of $f^{-1}(P_j)$ for varying $j$. 
\end{itemize}

If the sets $P_i$ are closed topological discs such that the intersection of any two is at most a finite union of topological intervals, then $G=\cup _{i=1}^r\partial P_i$ is a connected graph with finitely many  edges and vertices satisfying $G\subset f^{-1}(G)$. Conversely, given a  connected graph $G\subset f^{-1}(G)$ with finitely many edges and vertices, and such that all complementary components are topological discs, the  set of closures of the complementary components of $G$ is a Markov partition consisting of closed topological discs with boundaries which intersect in at most finitely many topological intervals. The construction we gave showed that a graph satisfying certain mild conditions can be approximated arbitrarily  closely by an isotopic graph $G'$ such that $G'\subset f^{-N}(G')$ for some integer $N$, and that some  $G\subset \cup _{i\ge 0}f^{-i}(G')$ satisfies all the above properties. 

In section 3 of \cite{R1}, we restricted to  a geometrically finite map $f$ of degree 2, for which the closures of the Fatou components were disjoint closed topological discs (as above), and  a one-dimensional parameter space $V_k$ of rational maps $g$  including $f$, with numbered critical points $c_1(g)$ and $c_2(g)$, with $c_1(g)$ of period $k$ under $g$, quotiented by M\"obius conjugacy preserving numbering of critical points.  We started with a graph $G$ with $G\subset f^{-1}(G)$  and investigated the associated partitions of parameter space, locally near   $f$. We were allowing $f$, there, to have  a parabolic cycle, but could equally well have considered a neighbourhood of the closure of a hyperbolic component, by starting from a hyperbolic map $f$, and making an initial choice of graph disjoint from the closures of periodic Fatou components. We showed in \cite{R1}  how to use backward iterates of $G$ to obtain successively finer partitions of the parameter space near $f$. In this paper we want to extend some aspects of this study. In particular, we want to consider {\em{type IV}} hyperbolic components in a neighbourhood in $V_k$ of the closure of the hyperbolic component of $f$, and to study the correspondence between these hyperbolic components and the periodic points of $f$. A {\em{type IV}} hyperbolic component, is a component of hyperbolic quadratic rational maps for which there are two distinct periodic cycles of Fatou components --- for which, of course, each cycle has to contain one of the two critical points. The first result of the paper concerning this is \ref{2.4}. This theorem is a characterisation of the hyperbolic components in terms of the symbolic dynamics, and complements the results of \cite{R1}, which associated symbolic dynamics to all maps admitting the Markov partition.

We also want to consider connections with known topological models of parts of parameter space, and, in particular, with matings. Mating was invented by Douady and Hubbard in the 1980's. It gives a way of constructing some rational maps, up to topological conjugacy,  and something more in some cases, from a pair of polynomials of the same degree, with locally connected Julia sets. In particular, a pair of hyperbolic critically finite quadratic polynomials which are not in conjugate limbs of the Mandelbrot set can be mated, and the mating is a critically periodic rational map, up to topological conjugacy.  It is known that not all type IV hyperbolic quadratic rational maps are represented up to topological conjugacy by matings, although there are some slices of parameter space, and some parts of slices of parameter space for which it is true. It is true in $V_2$, as is pointed out in one of the early results in \cite{Asp-Yam}. It is also true in some regions of $V_3$, as pointed out in \cite{R5}. One of the aims of the present paper is to consider the question locally, in particular, to consider hyperbolic components near the closure of a hyperbolic component, and ask whether all type IV hyperbolic components, which are sufficiently near, are conjugate to matings --- or are not all so conjugate. We shall find examples where all nearby type IV hyperbolic components are indeed represented by matings. These are new examples, as we shall be considering rational maps for which the closures of Fatou components are all disjoint. Our condition will be that the  lamination equivalence relation, associated to the mating  representing a hyperbolic critically finite  rational map --- of whose closure of hyperbolic component we are taking a neighbourhood --- is particularly simple. Counterexamples, for which it is definitely not true that all nearby hyperbolic components are represented by matings, are harder to prove completely, simply because there are a lot of matings available, and current techniques make it very hard to discount all of them. But there are plenty of candidates for which all nearby hyperbolic components are unlikely to be represented by matings, and for which the results of Section \ref{2} do give a description of all nearby hyperbolic components in terms of symbolic dynamics. 

This has a bearing on the question of continuity of mating.  Because of well-known, but unpublished, examples of Adam Epstein, it is known to be a problem to try to use this model on a space of complex dimension two, even near a geometrically finite map with a parabolic periodic point. But restricted to slices of complex dimension one, there is a chance that the combinatorial model is the right one in some parts of the parameter space. This is a question that has been considered and answered by Ma Liangang in his thesis \cite{Ma}, and being prepared for publication, under the sort of conditions imposed, and  by the sort of techniques used, in this paper. Rather comprehensive continuity  results have been proved in the very different  case of $V_2$ --- where the union of closures of Fatou components is $\overline{\mathbb C}$ ---  by Dudko \cite{Dudko}.
\newpage
\section{Basic set-up}\label{2}

\subsection{}\label{2.1}

We write $V_k$ for the set of quadratic rational maps $f$, with numbered critical points $c_1(f)$ and $c_2(f)$, quotiented by M\"obius conjugacy preserving numbering of critical points, and with $c_1(f)$ of period $k$. 
Our rational map $f$ will usually be a hyperbolic post-critically-finite quadratic rational map.  A hyperbolic quadratic post-critically-finite rational map is in $V_k$ for some $k\ge 1$, and is of {\em{type II, III or IV}}. Here, {\em{type II}}, means that both critical points are in the same periodic orbit,  {\em{type III}} means that one critical point, say $c_1(f)$, is periodic and the other, $c_2(f)$, is strictly pre-periodic, and  {\em{type IV}} means that the two critical points are in distinct periodic orbits. Any hyperbolic component intersecting $V_k$, for $k>1$, contains a unique post-critically-finite map, and is accordingly described as being of type II, III or IV.

We make the condition  that the closures of all Fatou components of $f$ are disjoint. Equivalently, we can stipulate that the closures of periodic Fatou components are disjoint. For $f\in V_k$, this is only possible if $k\ge 3$.

There are then many choices of  piecewise smooth graph $G_0=G_0(f)\subset \overline{\mathbb C}$, with finitely many edges and vertices, such that:
\begin{enumerate}[a)]
\item $G_0$ does not intersect the closure  of any periodic Fatou component of $f$, and the intersection of $G_0$ with any Fatou component is connected;
\item the components of $\overline{\mathbb C}\setminus G_0$ are all open topological discs, and their closures are all closed topological discs;
\item Any component of $\overline{\mathbb C}\setminus G_0$ contains at most one periodic Fatou component of $f$;
\item any edge of $G_0$ has two distinct endpoints, and exactly three edges of $G_0$, meet at any vertex;
\item the closures of any two components of $\overline{\mathbb C}\setminus G_0$ intersect in at most one edge of $G_0$, including endpoints.
\end{enumerate}

Theorem 1.2 of \cite{R1} then provides the existence of a graph $G'$ isotopic and arbitrarily close to $G_0$ such that $G'\subset f^{-N}(G')$, where $N$ is an integer which is bounded in terms of the degree of closeness required, and also provides the existence of a graph $G=G(f)$ with finitely many vertices and edges, satisfying  $G\subset \bigcup _{i\ge 0}f^{-i}(G')$, with $G\subset f^{-1}(G)$. Then the set $\mathcal{P}=\mathcal{P}(f)$ of closures of components of $\overline{\mathbb C}\setminus G$ is a Markov partition. It is standard, in dynamics, to encode points by admissible words in the elements of a Markov partition. A word $P_{i_0}\cdots P_{i_n}$, with $P_{i_j}\in \mathcal{P}$ for $0\le j\le n$, is {\em{admissible}} if $\bigcap _{j=0}^nf^{-j}(P_{i_j})\ne \emptyset $. Similarly, an infinite word $P_{i_0}P_{i_1}\cdots $ is admissible if $\bigcap _{j=0}^\infty f^{-j}(P_{i_j})\ne \emptyset  $. If $\mathcal{P}$ is a Markov partition, as defined here, then $P_{i_0}\cdots P_{i_n}$ is admissible if and only if $P_{i_j}\cap f^{-1}(P_{i_{j+1}})\ne \emptyset $ for all $0\le j<n$, and, similarly $P_{i_0}P_{i_1}\cdots $ is admissible if and only if $P_{i_j}\cap f^{-1}(P_{i_{j+1}})\ne \emptyset $ for all $j\ge 0$.

Under our hypotheses, any set in $\mathcal{P}$ contains at most one periodic Fatou component. We write $P^n(v_i)=P^n(f,v_i)$ for the set in $\mathcal{P}_n(f)$ which contains $v_i(f)=f(c_i(f))$.  Assuming that $f$ is hyperbolic post-critically-finite, either $v_2$ is periodic of some period $m$, or $v_2$ is in the backward orbit of $v_1$, in which case we write $m=k$.  In both cases, we can assume that $P^{j+1}(f,v_1)\subset \mbox{int}(P^j(f,v_1))$ for $0\le i\le k-1$ by ensuring a similar condition for the components of $\overline{\mathbb C}\setminus G_0(f)$ containing the Fatou components intersecting the forward orbit of $c_1$, which can itself be done, for example, by using the $\delta $-neighbourhood of the closure of these Fatou components for a suitable $\delta >0$, using the Poincar\'e metric. We then also have $P^{j+1}(v_1)\subset \mbox{int}(P^j(v_1))$ for all $j\ge 0$.  Similarly, we can ensure that $P^{j+1}(v_2)\subset \mbox{int}(P^j(v_2))$ for all $j\ge 0$. We will also choose $G_0$ so that the image under $f$ of any set in $\mathcal{P}(f)$ which contains a periodic Fatou component, also contains a single periodic Fatou component. The theorem  in \cite{R1} then ensures that the same is true for $G$.

  If $P_{i_0}P_{i_1}\cdots $ is an infinite admissible word, then any component of 
  $$\bigcap _{j=0}^\infty f^{-j}(P_{i_j})$$ is either a Fatou component or a point. If some component is a Fatou component, then the word must be eventually periodic, of the same eventual period as the Fatou component.  Further, under the given hypotheses, $\bigcap _{j=0}^\infty f^{-j}(P_{i_j})$ is either a finite union of Fatou components or is totally disconnected. At least in some cases, $\bigcap _{j=0}^\infty f^{-j}(P_{i_j})$ has more than one component. For let $z\in P^0(v_i)\setminus \{ v_i\} $. Write $P^0(v_i)=P_{i_1}$ and let $P_{i_0}$ be the set of ${\mathcal{P}}$ with $f^{-1}(P_{i_1})\subset P_{i_0}$. Let $P_{i_\ell }$ be the set of ${\mathcal{P}}$ with $f^{\ell-1}(v_i)\in P_{i_\ell}$, so that $z\in \bigcap _{\ell =0}^\infty f^{-\ell }(P_{i_{\ell +1}})$. Then the two elements $z_1$ and $z_2$ of $f^{-1}(z)$ are both in $\bigcap _{\ell =0}^\infty f^{-\ell}(P_{i_{\ell }})$.  

We can reduce the number of points represented by  at least some words, by considering words in $\mathcal{P}_n=\mathcal{P}_n(f)=\bigvee _{i=0}^nf^{-i}(\mathcal{P})$, where the sets in the partition $\mathcal{P}_n(f)$ are the components of non-empty sets $\bigcap _{j=0}^nf^{-j}(P_{i_j})$. If $Q_{i_0}\cdots Q_{i_p}$ is a word with letters in $\mathcal{P}_n$, and $Q_{i_j}\subset P_{i_j}\in \mathcal{P}$ and $Q_{i_p}$ is a component of $\bigcap _{j=0}^nf^{-j}(P_{i_{p+j}})$ then 
$$\bigcap _{j=0}^pf^{-j}(Q_{i_j})\subset \bigcap _{j=0}^{p+n}f^{-j}(P_{i_j}),$$
 and if $\bigcap _{j=0}^nf^{-j}(P_{i_{p+j}})$  is not connected, then $\bigcap _{j=0}^pf^{-j}(Q_{i_j})$ is strictly smaller than $\bigcap _{j=0}^{p+n}f^{-j}(P_{i_j})$.
 
\subsection{Variation of symbolic dynamics}\label{2.3}
 
It is a common idea in complex dynamics to look at variation of symbolic dynamics, although usually for specific Markov partitions, for example, the sequence of Markov partitions forming the Yoccoz puzzle. We start with a set $V_{k,0}\subset V_k$ such that a set $G_0(g)\cup \{ v_1(g),v_2(g)\} $ varies isotopically for $g\in V_{k,0}$, where $G_0(g)$ is a zero-level graph.  The maximal sets on which $G_n(g)$ varies isotopically form a partition of $V_{k,0}$ for each $n\ge 0$. This was considered in \cite{R1}. In this paper, we concentrate on type IV hyperbolic components. As one might expect, these are derived from periodic words in partition elements in some sense. In fact, we have the following theorem. We write $F_i(g)$ for the Fatou component for $g$ which contains $v_i(g)$, for $i=1$, $2$. 

\begin{theorem}\label{2.4} In terms of the given sequence of Markov partitions with the conditions previously specified, type IV post-critically-finite maps $g$ in $V_{k,0}$ are characterised by:
\begin{enumerate}[(i)] 
\item a finite sequence $f_i$, $1\le i\le N$,  of rational maps with $f_1=f$ and $f_N=g$;
\item integers $m_i$ for $1\le i\le N$ with $m_1=m$, and $j_i\ge 1$ for $1\le i<N$, and $v_2(f_i)$ is of period $m_i$, and 
\begin{equation}\label{2.4.1}m_{i+1}>\sum _{\ell =1}^ij_\ell m_\ell =n_\ell \end{equation}
 for $i\le N-2$, and either (\ref{2.4.1}) holds for $i=N-1$ also, or $m_N=j_{N-1}m_{N-1}$, $P^{n_{N-1}}(f_{N-1},v_2)$ varies isotopically to $P^{n_{N-1}}(f_{N},v_2)$ with $f_N^{jm_{N-1}}(v_2)\in P^{n_{N-1}}(f_{N},v_2)$ for all $0\le j\le j_{N-1}$, that is,  $f_N$ is a {\em{tuning}} of $f_{N-1}$;
\item integers $n_i$, for $0\le i\le N-1$, with $n_0=0$ and $n_i=n_{i-1}+j_im_i$ for $i>0$, such that there is an isotopy between $G_n(f_i)\cup \bigcup _{j=1}^kf_i^j(F_1(f_i))\cup \{ v_2(f_i)\} $ and $G_n(f_{i+1})\cup \bigcup _{j=1}^kf_{i+1}^j(F_1(f_{i+1}))\cup \{ v_2(f_{i+1})\} $ for $n\le n_i$, and between  $G_{n}(f_i)\cup \bigcup _{j=1}^kf_i^j(F_1(f_i))$ and $G_{n}(f_i)\cup \bigcup _{j=1}^kf_{i+1}^j(F_1(f_{i+1}))$ for $n\le n_i+m_{i+1}$, but this isotopy cannot be extended to map $v_2(f_i)$ to $v_2(f_{i+1})$ for $n=n_i+m_{i+1}$, and might not be possible for other $n$ with $n_i<n<n_i+m_{i+1}$ -- except in the case when $i=N-1$.
\end{enumerate}

\end{theorem}

\begin{proof}  As before, write $P^j(h,v_i)$ for the set of ${\mathcal{P}}_j(h)$ containing $v_i(h)$, for all $h\in V_{k,0}$. Since $v_1(h)$ is of period $k$ for all $h\in V_{k,0}$ we have 
$$h^k(P^{j+k}(h,v_1))=P^j(v_1)$$ for all $j\ge 0$, and since $v_2(f)$ is of period $m$, we have 
$$f^m(P^{j+m}(f,v_2))=P^j(f,v_2)$$
 for all $j\ge 0$. Now $f^m:P^m(f,v_2)\to P^0(f,v_2)$ is a covering of degree two.

Now let $g\in V_{k,0}$ be hyperbolic of type IV.  Let $S_{m,h}$ be the $2$-valued local inverse of $h^m$  defined on $P^0(h,v_2)$ for all $h\in V_{k,1}$,  with $S_{m,h}(P^0(h,v_2))=P^m(h,v_2)\subset P^0(h,v_2)$. There is a natural isotopy between $G_n(g)$ and $G_n(f)$ so long as $v_2(g)\in S_{m,g}^j(P^0(g,v_2))$. If this is true for all $j\ge 0$ then the period $m_2$ of $v_2(g)$ is $j_1m$ for some $j_1>1$, and we have $N=2$ and $f_2=g$.  If this is not the case, let  $j_1$ be the largest integer $j$ such that $v_2(g)\in S_{m,g}^j(P^0(g,v_2))$, that is, the largest integer $j$ with $S_{m,g}^j(P^0(g,v_2))=P^{mj}(g,v_2)$.  By hypothesis, we have $j_1\ge 1$. For any $h\in V_{k,1}$ with $v_2(h)\in S_m^{j_1}(P^0(h,v_2))$, we see that $G_{(j_1+1)m}(h)\cap P^{j_1m}(h,v_2)$ and $G_{(j_1+1)m}(f)\cap P^{j_1m}(f,v_2)$ are isotopic, and thus there is a natural correspondence between the sets of $\mathcal{P}_{(j_1+1)m}(h)$ in $P^{j_1m}(h,v_2)$ and the sets of $\mathcal{P}_{(j_1+1)m}(f)$ in $P^{j_1m}(f,v_2)$. For any such set $Q(h)$ in  $\mathcal{P}_{(j_1+1)m}(h)$, we write $Q(f,h)$ for the corresponding set in  $\mathcal{P}_{(j_1+1)m}(f)$. We should, of course, be careful with this. For example, $P^{m(j_1+1)}(g,f)$ is well-defined, but $P^{m(j_1+1)}(g,f)\ne P^{m(j_1+1)}(g,v_2)$. 

Since  $v_2(g)$ is periodic, there is a least integer $m_2>0$ such that $P^{m_2+j_1m}(g,v_2)$ is a component of $g^{-m_2}(P^{j_1m}(g,v_2))$. This is also the least integer $m_2$ such that $g^{m_2}(v_2(g))\in P^{j_1m}(g,v_2)$, and by the Markov properties we must have $g^i(P^{m_2+j_1m}(g))\cap P^{j_1m}(g,v_2)=\emptyset $ for $0<i<m_2$.  Then $m_2>j_1m$ because $g^{j_1m}$ maps $P^{j_1m}(g,v_2)\setminus P^{(j_1+1)m}(g,v_2)$ to $P^0(g,v_2)\setminus P^m(g,v_2)$. 

Now write $P^n(g,v_2)=P^n(g)$ for $0\le n\le m_2+j_1m$. Following the notation of \cite{R1}, we write $V(P^0,\cdots P^{m_2+j_1m};g)$ for the set of $h\in V_{k,1}$  for which  $(P^0(h,g),\cdots,P^{m_2+j_1m}(h,g))$ is {\em{equivalent}} to $(P^0(g,v_2),\cdots ,P^{m_2+j_1m}(g,v_2))$, meaning that there is a homeomorphism 
$$\varphi _{h,g}:G_{m_2+j_1m}(h)\cup \{ v_2(h)\} \to G_{m_2+j_1m}(g)\cup \{ v_2(g)\} ,$$
which therefore maps $P^n(h,g)$ to $P^n(g,v_2)$ for $n\le m_2+j_1m$.  In fact we can also choose $\varphi _{h,g}$ so that
$$\varphi _{h,g}\circ h=g\circ \varphi _{h,g}\mbox{ on }G_{m_2+j_1m}(h).$$
  By Theorem 3.2 of \cite{R1},  the set $V(P^0,\cdots P^{m_2+j_1m};g)$, which, of course, contains $g$, is connected. Also by Theorem 3.2 of \cite{R1} there exists $h\in V(P^0,\cdots P^{m_2+j_1m};g)$ such that $v_2(h)\in \bigcap _{j=0}^\infty S_{m_2,h}^{j}(P^{j_1m}(h,g))$, where $S_{m_2,h}$ is the branch of $h^{-m_2}$ which maps $P^{j_1m}(h,v_2)$ to $P^{m_2+j_1m}(h,g)$. We claim that there is such an $h$ with $v_2(h)$ of period $m_2$. One way to proceed is to use Theorem  3.2 of \cite{R1}, which says that  the map 
$$h\mapsto \varphi _{h,g}(h^{m_2}(v_2(h)))$$ 
maps $\partial V(P^0,\cdots P^{m_2+j_1m};g)$ homeomorphically onto $\partial P^{j_1m}(g)$. Hence
$$h\mapsto \varphi _{h,g}(h^{m_2}(v_2(h)))-\varphi _{h,g}(v_2(h))$$
is of degree one, and since $V(P^0,\cdots P^{m_2+j_1m};g)$ is contractible by Theorem 3.2 of \cite{R1} --- because the complement in $V_{k,1}$ is connected --- the map must have a zero, which is a point $h$ with $v_2(h)$ of period $m_2$. However it is found, Thurston's theorem for post-critically-finite branched coverings ensures that this map $h$ is unique. 

Now we write $f_{2}$ for this $h\in V(P^0,\cdots P^{m_2+j_1m};g)$. If $P^{j_1m+jm_2}(g,f_2)=P^{j_1m+jm_2}(g)$ for all $j\ge 0$ then the period of $v_2(g)$ is $j_2m_2$ for some $j_2>1$ and we have $N=3$ and $m_3=j_2m_2$. If not,  then  we can continue with $P^{j_1m}(g)=P^{n_1}(g)$ replacing $P^0$, with  $P^{m_2+j_1m}(g)$ replacing $P^m(g)$, and $m_2$ replacing $m$, and find $j_2\ge 1$, $m_3$, $f_{3}$, in the same way as $j_1$, $m_2$ and $f_2$ were found. Similarly, inductively, if $i<N$, we find $f_{i+1}$ from $f_i$ with $P^{n_i}(g)$ replacing $P^0(g)$ and $P^{n_i+m_{i+1}}(g)$ replacing $P^m(g)$. If $f_{i+1}$ is not a tuning of $f_i$ --- which is the case by definition if $i<N-1$ --- then we obtain (\ref{2.4.1}), because 
$$f_{i+1}^{j_\ell m_\ell}:P^{n_\ell }(f_{i+1})\to P^{n_{\ell -1}}(f_{i+1})$$
for $1\le \ell \le i+1$ and if 
$$p_{\ell,i} =\sum _{n=\ell }^{i+1}j_nm_n$$ then 
$$f_{i+1}^{p_{\ell,i}}(v_2)\in P^{n_{\ell -1}}(f_{i+1})\setminus P^{n_{\ell -1}+m_{\ell }}(f_{i+1}).$$
Therefore, since 
$$p_{0,i}=\sum _{i=0}^\ell j_\ell m_\ell $$ 
and 
$$f_{i+1}^{p_{i,0}}(v_2)\notin P^{m_{i+1}}(f_{i+1}),$$
we have (\ref{2.4.1}), as required.

\end{proof}
We also have the following counting result.

\begin{theorem}\label{2.5} Hyperbolic components in $V(P^0;f)\subset V_{k,0}$, for which $F_2(g)$ is of period $n$, are in two-to-one correspondence with points in $P^0(f)$ of period $n$.\end{theorem}

\begin{proof} The number of periodic points of period $n$ in $P^0(g)$ is the same for all $g\in V_{k,0}$ which do not have a parabolic point of period $n$ and multiplier $1$. Now $g^{-1}(G_\ell (g))$ varies isotopically for $g\in V(P^0,P^1,\cdots P^\ell;g_0)$ for any fixed $g_0$. So the possible sets $P^{\ell +1}(g)\subset P^\ell (g)$ vary isotopically for all $g\in V(P^0,P^1,\cdots,P^\ell;g_0)$, and the number of periodic points of period $n$ in $P^{\ell +1}(g)$ is the same for all $g\in V(P^0,P^1,\cdots P^\ell ;g_0)$ which do not have a period $n$ point in $P^{\ell +1}$ with multiplier $1$. Any set $\partial P^\ell(g) \setminus \partial P^0(g)$ contains no periodic points, as it is contained in $G_\ell(g)\setminus G_0(g)=g^{-\ell}(G_0(g))\setminus G_0(g)$. Hyperbolic components of period $n$ which are in copies of the Mandelbrot set of lower period correspond to periodic points of period $n$ in sets $\bigcap _{i=0}^\infty P^{ik}(g)$ where $P^{(i+1)k}(g)$ is a component of $g^{-k}(P^{ik}(g))$, for $k$ properly dividing $n$, and $k$ is the largest possible such integer. For such points, the two-to-one correspondence between the number of points of period $n$ and the number of hyperbolic components of period $n$ follows from the well-established fact that there are $2^{n/k-1}$ hyperbolic components of period dividing $n/k$ in the Mandelbrot set -- while the number of points of period dividing $n/k$ for any non-parabolic parameter value is always $2^{n/k}$. So in order to complete the proof, it suffices to prove that if $V(P^0,P^1,\cdots P^n;g_0)$ contains a hyperbolic component of period $n$ which is not in a copy of the Mandelbrot set attached to a hyperbolic component of lower period dividing $n$, then it contains only one, and $P^n(g)$ contains two points of period $n$ for non-parabolic parameter values. We see this as follows. The condition about not being in a copy of the Mandelbrot set ensures that $P^n(g)$ is mapped to $P^0(g)$ by $g^n$ with degree two, and $P^n(g)$ is not contained in $g^i(P^n(g))$ for $0<i<n$, and then the number of points in $P^n(g)$ fixed by $g^n$ is two, because this is the number of fixed points (up to multiplicity) for a  holomorphic degree two map of a disc onto itself, by Rouch\'e's Theorem. There are, in fact, exactly two fixed points, as $g$ does not have parabolic fixed points. 

\end{proof}

Now we want to investigate some ways in which the configuration of hyperbolic components near $f$ is affected by  the topology and dynamics of the sequences of graphs and Markov partitions. We are particularly interested in the case when $f$ is represented by a mating. For the results in Section \ref{3}, we will need a variation on the result of \ref{2.4} which is proved in exactly the same way. So we simply state the result, with a brief note on the proof. Note that the first hypothesis is exactly the same as in \ref{2.4}.

\begin{theorem}\label{2.6} Let $g\in V_{k,0}$ be postcritically finite of type IV sufficiently near $f$, that is $v_2(g)\in P^{j_1m}(g,f)$ for $j_1$ sufficiently large. Let $n_1=j_1m$, $m_1=m$, and let $i_1$ be a fixed integer. Let  $r_i$, $n_i$, be defined for $0\le i\le t$ and let $m_i$, $j_i$ be defined for $1\le i\le t$ satisfying  
$$n_0=r_0=0,$$
$$r_1=i_1m,$$
$$n_1-r_1=(j_1-i_1)m,$$
$$n_2-r_2=(j_2-1)m_2+m_2'$$
for $m_2'=i$ for the greatest $i<m_2$ with  $g^{i}(v_2)\in P^{r_1}(g)$,
$$n_i-r_i\ge n_{i-1}\mbox{ if }i\ge 3,$$
$$r_i<m_i\mbox{  if  }i>1,$$
$$j_i\ge 1,$$
and for $0\le i<t$:
\begin{equation}\label{2.6.1}n_{i+1}=n_i-r_i+j_{i+1}m_{i+1} ,\end{equation}
\begin{equation}\label{2.6.2}n_i-r_i+m_{i+1}>n_i.\end{equation}

Suppose also that rational maps $f_i$  have been defined with
$$f_i\in V(P^0,\cdots P^{n_i},g),$$
$v_2(f_i)$ is of period $m_i$,
$$f_{i+1}\in V(P^0,\cdots P^{n_i},f_i).$$
Let $S_{m_i,h}$ be the local inverse of $h^{m_i}$ mapping $h^{m_i}(v_2(h))$ to $v_2(h)$, on some appropriate domain.

Suppose also that, for $1\le i\le t$, the set $P^{n_i-r_i+,1}(h)$ is  defined for $h\in V(P^0,\cdots P^{n_i-r_i},g)$, is a  topological disc and a union of sets of ${\mathcal{P}}_{n_i-r_i+m_2-m_2'}(h)$, with
$$P^{n_1-r_1,1}(h)=P^{n_1-r_1}(h),$$
and for $i\ge 1$ and $0\le j<j_{i+1}$,
\begin{equation}\label{2.6.6}P^{n_i-r_i,1}(h)\subset P^{n_i-r_i}(g),\end{equation}
and for $i\ge 3$,

\begin{equation}\label{2.6.5}g^j(v_2(g))\notin P^{n_i-r_i,1}(g),\ j\le j_{i}m_i,\ j\ne m_{i}.\end{equation}

Here $j_{i+1}$ is the greatest integer such that  
$$v_2(g)\in S_{m_{i+1},g}^{j_{i+1}}(P^{n_{i}-r_{i},1}(g))=P^{n_{i+1},1}(g)$$
and $m_{i+1}$ is the least integer $>m_i$ with 
$$g^{m_{i+1}}(v_2(g))\in P^{n_i-r_i,1}(g)$$

Then  we can define $m_{t+1}$ , $n_{t+1}$  similarly to the above with $j_tm_t<m_{t+1}$, and can find $f_{t+1}$ with corresponding properties to the $f_i$ for $i\le t$.\end{theorem}

Note that $m_i$ is an increasing sequence by the condition $m_{i+1}>j_im_i$, and so is $n_i$, by (\ref{2.6.1}) and (\ref{2.6.2}). There is work to do to establish the hypotheses, which will be done in the later sections, but here is the proof of \ref{2.6}, subject, of course, to the hypotheses.
\begin{proof}  
Our condition (\ref{2.6.5}) and the definition of $j_i$ ensure that $S_{m_{i+1},g}$ is two-valued on $P^{n_i-r_i,1}(g)$ for $i\le t$, because, by (\ref{2.6.5}), if $g^rS_{m_{i+1},g}$ is not two-valued for some $1\le r<m_{i+1}$ then we must have $m_{i+1}-r=m_i$ and $g^rS_{m_{i+1},g}=S_{m_i,g}$ and hence 
$$g^r(S_{m_{i+1},g}(P^{n_i-r_i,1}(g)))\subset P^{n_i-r_i,1}(g).$$
But this implies that $g^r(v_2(g))\in P^{n_i-r_i,1}(g)$, which, by (\ref{2.6.5}), implies that $r=m_i$ and hence 
$$S_{m_{i+1},g}=S_{m_i,g}^2.$$
But then by (\ref{2.6.6}),
$$v_2(g)\in S_{m_i,g}^2(P^{n_i-r_i,1}(g))\subset S_{m_i,g}(P^{n_i,1}(g)),$$
contradicting the definition of $j_i$. So $S_{m_{i+1},g}$ is indeed two-valued on $P^{n_i-r_i,1}(g)$ and hence also $S_{m_{i+1},h}$ is two-valued  for all $h\in V(P^0,\cdots P^{n_i-r_i,1},g)$. So  exactly as in \ref{2.4}  there is a unique $h\in V(P^0,\cdots P^{n_i-r_i},g)$ with 
 $$v_2(h)\in \bigcap _{j\ge 0}S_{m_{i+1},h}^jP^{n_{i}-r_{i}}(h)$$
 and $h^{m_{i+1}}(v_2(h))=v_2(h)$. \end{proof}

\section{Quadratic Polynomials, Laminations, Matings, and the Main Theorem about Matings}\label{3}
\subsection{}\label{3.1}

Quadratic polynomials  do not fit into the framework used in Section \ref{2}, unless there is a single Fatou component, because the Julia set is the boundary of a single Fatou component, the one containing infinity, and therefore this Fatou component has boundary in common with any other Fatou component. But, of course, it is easy to construct Markov partitions for the Julia set of a quadratic polynomial in the Mandelbrot set. For example,  for a polynomial which is not in the main cardioid, one can simply take the zero-level Yoccoz partition, using the rays landing at the $\alpha $ fixed point. 

We recall the combinatorial description of hyperbolic polynomials in the Mandelbrot set, originally realised by Douady and Hubbard \cite{D-H1} and reinterpreted by Thurston \cite{T}. The Mandelbrot set is the set of quadratic polynomials $f_c(z)=z^2+c$ ($c\in \mathbb C$) with connected Julia sets, or equivalently, the orbit of the critical point $0$ remains bounded, or, equivalently, the unbounded Fatou component $U_\infty (c)$, including $\infty $, is conformally equivalent to the open unit disc $D=\{ z:|z|<1\} $. Under any of these conditions, there is a holomorphic bijection 
$$\varphi _c:\{z\in \mathbb C:|z|>1\} \to U_\infty (c)$$
satisfying 
\begin{equation}\label{3.1.1}\varphi _c(z^2)=f_c\circ \varphi _c(z),\ \ |z|>1.\end{equation}
If $f_c$ is hyperbolic, and more generally, if the Julia set  $J(f_c)$ is locally connected, then the map $\varphi _c$ extends to map the unit circle $S^1=\{ z\in \mathbb C:|z|=1\} $ to the Julia set $\partial U_\infty (c)$, with (\ref{3.1.1}) holding for $|z|=1$ also. Then we define $\mbox{conv}(\varphi _c^{-1}(z))$ to be the convex hull in $\overline{D}$  of $z\in J(f_c)$. It does not greatly matter which metric is used on the unit disc to define the convex hull, provided that geodesics in the metric exist and intersect in at most one point. Usually either the Euclidean or hyperbolic metric  is used.  Thurston \cite{T} made the simple but crucial observation that convex hulls are either equal or disjoint.  The set $\varphi _c^{-1}(z)\subset S^1$ is always finite, and so the convex hull is either a point  a finite-sided polygon, and $\partial \varphi _c^{-1}(z)$ is either a single point or a finite union of geodesics, with endpoints. The set 
$$L_c=\{ \ell :\ell \subset \partial \varphi _c^{-1}(z)\mbox{ is a geodesic, with endpoints }\} $$
is a {\em{clean quadratic invariant lamination}}, of which the {\em{leaves}} are the geodesics with endpoints, $\ell $, as above. The invariance properties are given in terms of the endpoints of leaves.

\begin{itemize}
\item {\em{Forward invariance.}} If $\ell \in L_c$ is a leaf with endpoints $z_1$, $z_2$, then either $z_2=-z_1$ or there is a leaf in $L_c$ with endpoints $z_1^2$ and $z_2^2$, which is called $\ell ^2$
\item {\em{Backward invariance.}} If $\ell \in L_c$ is a leaf with endpoints $z_1$ and $z_2$, then there are two leaves of $L_c$, one with endpoints $z_3$ and $z_4$, and the other with endpoints $-z_3$ and $-z_4$, such that $z_3^2=z_1$ and $z_4^2=z_2$.
\item {\em{Clean.}} If two leaves of $L_c$ have a common endpoint, then the two leaves are both in the boundary of the same component of $\overline{D} \setminus \cup L_c$, called a {\em{gap}} of $L_c$.
\end{itemize}

The intersection of the lamination with $D$ is closed in  $D$. The {\em{length}} of a leaf is defined in terms of its endpoints $e^{2\pi ia}$ and $e^{2\pi ib}$ with $0\le a<b<1$, as $\mbox{min}(b-a,1-(b-a))$, so that the length is always $\le \frac{1}{2}$. If $f_c$ and $f_{c'}$ are in the same hyperbolic component then $L_c=L_{c'}$. If $f_c$ is hyperbolic, then the longest leaf has length $<\frac{1}{2}$ and is periodic under the map $\ell \mapsto  \ell ^2$. The endpoints of $\ell $ are, of course, also periodic under the map $z\mapsto z^2$, of the same period as $\ell $, or twice the period of $\ell $. The period of the endpoints is the same as the period of the bounded Fatou components of $f_c$. If $f_c$ is hyperbolic, and in some other cases, one can use the lamination $L_c$ to reconstruct $f_c$ up to topological conjugacy on $\overline{\mathbb C}$. We can extend $\varphi _c$ to map $\overline{\mathbb C}$ to $\overline{\mathbb C}$ by mapping leaves  and finite-sided gaps to the same points as their endpoints are mapped to, and mapping each infinite-sided gap $G$ to the Fatou component whose boundary is the image under $\varphi _c$ of the boundary of $G$. There is such a Fatou component, and it is unique. We can define a branched covering $s_{L_c}:\overline{\mathbb C}\to \overline{\mathbb C}$ to be $z^2$ on $\{z:|z|\ge 1\} $, and to map leaves of $L_c$ to leaves, finite-sided gaps homeomorphically  to finite-sided gaps, and infinite-sided gaps either homeomorphically or by degree-two branched coverings to infinite-sided gaps. In addition we can ensure that on  any periodic cycle of infinite-sided gaps, $\varphi _c$ is a conjugacy between $s_{L_c}$ and $f_c$. In order to define $s_{L_c}$  uniquely up to topological conjugacy, we make this conjugacy on periodic cycles of infinite-sided gaps only  when $f_c$ is postcritically finite, and then define $s_{L_c}$ near critical forward orbits so as to be locally conjugate to $f_c$. So in such cases, $s_{L_c}$ is a postcritically finite branched covering. If we use the extension  $\varphi _c:\overline{\mathbb C}\to \overline{\mathbb C}$, then $\varphi _c\circ s_{L_c}=f_c\circ \varphi _c$ on $\overline{\mathbb C}$. Since $\varphi _c$ maps leaves of $L_c$ to points, it is a semi-conjugacy, unless $L_c$ is empty, that is, unless $c=0$. 

Amazingly, as noted by Douady-Hubbard and Thurston (the language used here is that of Thurston), this description can be completely reversed. A clean quadratic invariant lamination $L$ is uniquely determined by its {\em{minor leaf}}, the image of its one or two  longest leaves -- two of them unless the longest leaf has length $\frac{1}{2}$. For simplicity, we assume that the longest leaf of $L$ has length $<\frac{1}{2}$ and is a side of an infinite-sided gap. Then this longest leaf is periodic, and $L=L_c$ for some hyperbolic critically periodic polynomial $f_c$ with $c\ne 0$. This means that $0$ has some period $>1$ under $f_c$. The period $k$ of the endpoints of the minor leaf of $L_c$ is the same as the period $k$ of $0$ under $f_c$. The endpoints are therefore of the form $e^{2\pi ix_j}$, $j=1$, $2$, where $x_j$ is of period $k$ under the map $x\mapsto 2x\mbox{ mod }1$, that is, $x_j$ is a rational of odd denominator, where the denominator is a divisor of $2^k-1$ but not of $2^{k_1}-1$ for any $k_1<k$. Each such $x_j$ is the endpoint of precisely one minor leaf of a clean invariant lamination. Thus, such minor leaves are in one-to-one correspondence with hyperbolic components of quadratic polynomials, in the Mandelbrot set, but excluding the main cardioid (which is the hyperbolic component of $f_0(z)=z^2$). Moreover such minor leaves are disjoint, and the closure of the set of such minor leaves is a lamination of minor leaves called $QML$. Conjecturally, the quotient of $\overline{D}$ by the associated equivalence relation, in which the equivalence classes are either leaves of $QML$, or closures of finite-sided gaps of $QML$, is homeomorphic to the Mandelbrot set.

We often write $L_{x}$ for the lamination with minor leaf with endpoint at $x$, and $s_x$ for $s_{L_x}$. Thus, $L_{x_1}=L_{x_2}$ if $e^{2\pi ix_1}$ and $e^{2\pi ix_2}$ are endpoints of the same minor leaf, and $s_{x_1}=s_{x_2}$. The minor leaf is denoted by $\mu _{x_1}$ or $\mu _{x_2}$, as, of course, $\mu _{x_1}=\mu _{x_2}$.
Minor leaves $\mu _1$ and $\mu _2$ are partially ordered by: $\mu _1<\mu_2$ if $\mu _1$ separates $\mu _2$ from $0$ in the closed unit disc. Every minor leaf is either minimal in this ordering or is separated from $0$ by a minimal minor leaf. The minimal minor leaves are the minors of the hyperbolic components adjacent to the main cardioid of the Mandelbrot set.  The set of minor leaves bounded from $0$ by a fixed minimal minor leaf is called a {\em{combinatorial limb}}. 

As well as polynomials, laminations $L$ and lamination maps $s_L$ can be used to define some rational maps, up to topological conjugacy. In fact, for sufficiently general laminations, all hyperbolic quadratic rational maps can be described in this way \cite{R6,R7}, but here we concentrate on rational maps which are represented by {\em{matings}}. If $p$ and $q$ are any odd denominator rationals then we can define a postcritically finite branched covering  $s_p\Amalg s_q$ by
$$s_p\Amalg s_q(z)=\begin{array}{ll}s_p(z)&\mbox{ if }|z|\le 1,\\
(s_q(z^{-1}))^{-1}&\mbox{ if }|z|\ge 1.\end{array}$$
This definition is consistent, since $s_p(z)=s_q(z)=z^2$ if $|z|=1$, and $(z^{-1})^2=(z^2)^{-1}$. We also define 
$$L_q^{-1}=\{ \ell ^{-1}:\ell \in L_q\} $$
where $\ell ^{-1}=\{ z^{-1}:z\in \ell \} $. We have the following theorem. This is a combination of  two theorems. The first is Tan Lei's theorem \cite{TL} on a necessary and sufficient condition for $s_p\Amalg s_q$ to be Thurston-equivalent to a rational map (proved more generally in \cite{R6, R7}), which is itself verification that there is no Thurston obstruction, and hence a deduction from Thurston's theorem for Post-critically-finite Branched Coverings \cite{D-H2,R3}. The second theorem is a deduction from the results of  Chapter 4 of \cite{R6} which show that Thurston equivalence to a rational map $f$ implies semi-conjugacy, that is, there is a continuous surjection $\varphi $ such that $\varphi \circ (s_p\Amalg s_q) =f\circ \varphi $,  and any set $\varphi ^{-1}(x)$ is either a single point, or a connected union of finitely many leaves and finite-sided gaps of $L_p\cup L_q^{-1}$. 

{\em{Thurston equivalence}} is the appropriate definition of homotopy equivalence for {\em{post-critically-finite}} branched coverings.  The {\em{postcritical set}} $X(f)$ of a branched covering $f$ is defined by 
$$X(f)=\{ f^n(c):n>0,\ \ c\mbox{ critical }\}  .$$
The branched covering $f$ is {\em{post-critically-finite}} if the set $X(f)$ is finite. Two post-critically-finite branched coverings $f_0$ and $f_1$ are {\em{Thurston equivalent}}, written $f_0\simeq f_1$, if there is a homotopy $f_t$ through post-critically-finite branched coverings from $f_0$ to $f_1$ such that $X(f_t)$ varies isotopically from $f_0$ to $f_1$. Equivalently, there is a homeomorphism $\varphi $ and a homotopy $g_t$ ($t\in [0,1]$) through post-critically-finite branched coverings such that $X(g_t)$ is the same set for $t\in [0,1]$, and $\varphi \circ f_0\circ \varphi ^{-1}=g_0$, and $g_1=f_1$. Also equivalently, there are homeomorphisms $\varphi $ and $\psi $ such that  $\varphi \circ f_0\circ \psi ^{-1}=f_1$ and $\varphi (X(f_0))=\psi (X(f_0))=X(f_1)$, and $\varphi $ and $\psi $ are isotopic via an isotopy which is constant on $X(f_0)$. 

\begin{theorem}\label{3.2} Let $p$ and $q$ be odd denominator rationals. Let $\sim _{p,q}$ be the equivalence relation generated by two points being equivalent if and only if they are in the same leaf of $L_p\cup L_q^{-1}$, or in the closure of the same finite-sided gap of $L_p\cup L_q^{-1}$. Let $[s_p\Amalg s_q]$ denote the map on the quotient space $\overline{\mathbb C}/\sim _{p,q}$. Then the following are equivalent.
\begin{itemize}
\item[1.] $\mu _p$ and $\mu _q$ are not in conjugate combinatorial limbs.
\item[2.] $s_p\Amalg s_q$ is Thurston equivalent to a post-critically-finite hyperbolic quadratic rational map.
\item[3.] $\overline{\mathbb C}/\sim _{p,q}$ is homeomorphic to $\overline{\mathbb C}$, each equivalence class of $\sim _{p,q}$ is a point , or a finite union  of at most $N$ leaves and closures of finite-sided gaps, for some $N$ depending on $(p,q)$, and  $[s_p\Amalg s_q]$ is topologically conjugate to a post-critically-finite hyperbolic rational map.
\end{itemize}
\end{theorem}

The statement of the following result has  nothing to do with the theory above, but quadratic invariant laminations play an important role in the proof, as we shall see. 

\begin{proposition}\label{3.3}  Let $f$ be quadratic rational hyperbolic type IV with periodic Fatou components $U_1$ and $U_2$ of periods $k_1$ and $k_2$ under $f$, containing the critical values of $f$. Then one of the following occurs.

\begin{itemize}
\item[1.] $k_i=1$. 
\item[2.] The closure of any Fatou component in the full orbit of $U_i$ is a closed topological disc. The closures  of any two Fatou components in the full orbit of $U_i$ intersect in at most one point, and any such point is in the full orbit of the unique point in $\partial U_i$ which is fixed by $f^{k_i}$.
\end{itemize}
 Also one of the following occurs.
\begin{itemize}
\item[3.] Either $k_1=1$ or $k_2=1$. 
\item[4.] The closures of any two Fatou components intersect in at most one point, and if $x$ is such a point, there are $n\ge 0$ and $i\in \{ 1,2\} $ such that $f^n(x)$ is the unique point in $\partial U_i$ fixed by $f^{k_i}$. \end{itemize}
\end{proposition}

\begin{proof}  Let $\varphi _{i,j}:S^1\to f^j(\partial U_i)$ be the unique orientation-preserving continuous map, extending to map the open unit disc $D$ homeomorphically 
 to $f^j(U_i)$, with $\varphi _{i,j}(z^2)=f^{k_i}\circ \varphi _{i,0}(z)$ for all $z\in S^1$. Write $\varphi _{i,0}=\varphi _i$. Thus, $\varphi _{i,j}=f^j\circ \varphi _i$ on $S^1$, for $0\le j<k_i$. We also write $s(z)=z^2$.
 
 If $W_1$ is a Fatou component,  then mapping forward under $f$, we have $f^n(W_1)=U_i$ for some $n\ge 0$ and for $i=1$ or $2$. If $\overline{W_1}$ is not a closed topological disc, then $f^\ell (\overline{W_1})$ is not a closed topological disc, for all $\ell \ge 0$. If $W_2$ is another Fatou component, such that  $\partial W_1$ and $\partial W_2$ intersect in $x$, and  $f^j(W_1)\ne f^j(W_2)$ for some $0\le j<n$ but $f^{j+1}(W_1)=f^{j+1}(W_2)$, then $f^j(\partial W_1)$ and $f^j(\partial W_2)$ must intersect in at least two points and $f^{\ell }(\overline{W_1})$ is not a closed topological disc for $j+1\le \ell \le n$.  

 So we will show that, if $\overline{U_i}$ is not a topological disc, or $U_i$ shares at least two  boundary points with some other Fatou component in its backward orbit, then $k_i=1$, and if $U_i$ shares  at least two boundary points with some other Fatou component, then at least one of $k_1$ and $k_2=1$.
 
 The idea of the proof is as follows.  We will construct a quadratic forward invariant lamination $L_i$ on $\overline{D}$  which is nonempty if and only if  either \\ $\varphi _i:S^1\to \partial U_i$ is not injective, or $U_i$ has at least two boundary points in common with another Fatou component. Also, $L_i$ does not have a diagonal leaf. Therefore, if $L_i$ is nonempty, it has at least one periodic leaf. Then we can construct a  critically periodic {\em{tuning}} $f_1$ of $f$, which is a nontrivial tuning of the forward orbit of $v_i(f)$ if $L_i\ne \emptyset $,  which preserves $\Lambda $, where $\Lambda $ is the union over $i$ and $j$ of images under the maps  $\varphi _{i,j}$ of some finite set of closures of periodic leaves of $L_i$, and $\Lambda $ is also a union of closed loops.  Then $\Lambda $ contains a Levy cycle for $f_1$. Then the Tuning Proposition 1.20 of \cite{R6} implies that $k_i=1$  if $f_1$ is a nontrivial tuning round the forward orbit of $v_i$, that is, if $L_i\ne \emptyset$. This is because the Tuning Proposition says that a nontrivial tuning of a type IV postcritically finite rational map does not have Levy cycles, except sometimes if the periodic cycle being tuned has period $1$. The Tuning Proposition needs to be applied twice, to each of the postcritical orbits of $f$ in succession, if $f_1$ is a nontrivial tuning round both the postcritical orbits of $f$. For a set $A\subset S^1$, we write $C(A)$ for the convex hull in $\overline{D}$ (in the Euclidean metric, for the sake of concreteness). 

First  suppose that $\overline{U_i}$ is not a closed topological disc. Then we define $L_i$ to be the set of leaves $\ell $ such that $\ell $ is a component of $D\cap C(\varphi _i^{-1}(x))$ for some $x\in \partial U_i$ with $\#(\varphi _i^{-1}(x))\ge 2$. The  sets $C(\varphi _i^{-1}(x))$, for different $x$, are disjoint. This is a familiar argument from \cite{T}, deriving from the fact that two smooth loops on the sphere which intersect transversally can never have a single intersection. So the leaves of $L_i$ have disjoint interiors. $L_i$ is closed (by continuity of $\varphi _i$) and clean, by definition. Also $L_i$ does not contain a diagonal leaf, because if $\varphi _i^{-1}(x)$ contains a diagonal leaf, $x$ must be a critical point of $f^{k_i}$. So 
$$s(C(\varphi _i^{-1}(x))\cap S^1)=C(\varphi _i^{-1}(f^{k_i}(x)))\cap S^1,$$
and $L_i$ is forward invariant. The closure of every leaf in $L_i$ is a closed loop.  Let $L_{i,b}$ be the unique quadratic invariant lamination such that the forward orbit of every leaf in $L_{i,b}$ intersects $L_i$ and at least one of the longest leaves of $L_{i,b}$ is in $L_i$. Since $L_i$ is closed, $L_{i,b}$ is also closed. Then $L_{i,b}$ either has a fixed leaf with endpoints of period $2$ or a fixed $r$-sided gap with sides of period $r$ for some finite $r$.  These periodic leaves are also in $L_i$, and therefore $L_i$ has at least one periodic leaf $\ell $. Let $f_1$ be a critically periodic  tuning of $f$ round the orbit of $v_i(f)$ which preserves up to isotopy the set 
$$\Lambda =\bigcup _{j=0}^{k_i-1}\bigcup _{t\ge 0}\varphi _{i,j}(s^t(\overline{\ell })).$$
Then $\Lambda $ contains a Levy cycle for $f_1$, giving the required contradiction. So $k_i=1$.

Now suppose that $\overline{U_i}$ is a closed topological disc, and $U_i$ shares at least two boundary points with $U_3$, where $U_3$ is in the full orbit of $U_i$. Then $f^n(U_i)\ne f^n(U_3)$ for all $n\ge 0$, because otherwise $f^n(\overline{U_i})$ is not a closed topological disc, for any $n$ such that $f^n(U_i)=f^n(U_3)$. Then we define $L_i$ to be the set of all leaves such that $\ell $ is a component of  $D\cap \partial C(\varphi _i^{-1}(\partial U_i\cap \partial U_3))$ for some such  Fatou component $U_3$. Once again, the  sets $C(\varphi _i^{-1}(\partial U_i\cap \partial U_3))$, for varying $U_3$, have disjoint interiors, by the same argument from \cite{T}. Also, $L_i$ is closed in $D$,  since only finitely many Fatou components have spherical diameter $>\varepsilon $, for any $\varepsilon >0$. There is no diagonal leaf, because that would imply $\overline{U_i}$ was not a topological disc. Forward invariance of $L_i$ follows from 
$$s(C(\varphi _i^{-1}(\partial U_i\cap \partial U_3)))\cap S^1=C(\varphi _i^{-1}(\partial U_i\cap \partial f^{k_i}(U_3)))\cap S^1.$$
So once again $L_i$ contains a periodic leaf $\ell $. This time $\varphi _i(\ell )$ shares endpoints with $\varphi _{i,j}(\ell ')$ for some $\ell '\in L_i$ of the  same period, and $\varphi _i(\overline{\ell} )\cup\varphi _{i,j}(\overline{\ell '})$ is a closed loop. As before we can make a tuning  $f_1$ of $f$ round the forward orbit of $v_1(f)$, which admits a Levy cycle, leading to a contradiction unless $k_i=1$.

Finally suppose that $\overline{U_i}$ is a closed topological disc for both $i=1$ and $i=2$,  and that $U_1$ has at least two boundary points in common with $f^j(U_2)$ for some $j\ge 0$. Then, in the same way as above, we can construct nonempty forward invariant laminations $L_i$ of leaves $\ell $ such that $\ell $ is a component of  $D\cap \partial C(\varphi _i^{-1}(\partial U_i\cap \partial U_3))$ for some Fatou component $U_3$ not in the full orbit of $U_i$. Then once again we can construct a tuning $f_1$ of $f$, this time round both the postcritical orbits of $f$, and preserving a Levy cycle where each loop is of the form $\varphi _{1,j_1}(\overline{\ell _1})\cup \varphi _{2,j_2}(\overline{\ell _2})$ for some $\ell _1\in L_1$ and $\ell _2\in L_2$. Then the Tuning Proposition gives $k_1=1$ or $k_2=1$. 
 
Now we consider the the extra properties which hold in the alternatives 2  and 4 of the proposition.

So suppose that $k_i>1$.  If  $U_i$ has a boundary point fixed by $f^{k_i}$, which is $\varphi _i(z)$ for some $z\ne 1$,  then  $\varphi _i(z)=f^{k_i}(\varphi _i(z))=\varphi _i(z^2)$ and $\overline{U_i}$ is  not  a closed topological disc.  So if $U_i$ has a boundary point $\varphi _i(z)$ in common with $f^j(U_i)$ for some $0<j<k_i$, with $z\ne 1$  and $\overline{U_i}$ is a closed topological disc, then   $U_i$ and $f^j(U_i)$ have the distinct boundary points $\varphi _i(z)$ and $\varphi _i(z^2)$ in common. We have seen that this does not happen. Therefore $z=1$. This completes the proof of the dichotomy 1 or 2.

Now we consider the dichotomy of 3 or 4 of the lemma. So suppose $k_1>1$ and $k_2>1$.  If $U_1$ and $f^j(U_2)$ have a boundary point $x=\varphi _1(z_1)=\varphi _{2,j}(z_2)$ in common, with $z_1\ne 1$ and $z_2\ne 1$, then since  both $\overline{U_1}$ and $f^j(\overline{U_2})$ are closed topological discs, the period of $x$ under $f$ is a proper multiple of both $k_1$ and $k_2$: $\ell _1k_1=\ell _2k_2$ where $\ell _i$ is the period of $z_i$ under $s$. Then $U_1$ has boundary point $\varphi _1(s^t(z_1))$ in common with $f^{j+k_1t}(U_2)$ for $0\le t<\ell _1$. If there is $t<\ell _1$  such that $k_2$ divides $k_1t$, then  $\varphi _1(s^t(z_1))$ is a second common boundary point of $U_1$ and $f^j(U_2)$, which we have already discounted. So now we assume that $k_1\ell _1=k_2\ell _2$ is the least common multiple of $k_1$ and $k_2$ and that $\overline{U_i}$ is a closed topological disc for both $i=1$ and $i=2$, and any two sets $\partial U_1$ and $f^j(\partial U_2)$ intersect in at most one point. This means that if $d$ is the greatest common divisor of $k_1$ and $k_2$, then $k_1=d\ell _2$ and $k_2=d\ell _1$, and the greatest common divisor of $\ell _1$ and $\ell _2$  is $1$. Now if $f^{j_1}(U_1)$ shares a boundary point with $f^{j_2}(U_2)$, then $f^{j_1}(U_1)$ shares a boundary point with $f^{j_3}(U_2)$ if and only if $d$ divides $j_3-j_2$. It follows that if 
 $$A_j=\{ \ell :f^j(\partial U_1)\cap f^\ell (\partial U_2) \ne \emptyset\} $$
 then any two sets $A_{j_0}$ and $A_{j_1}$ are either equal or disjoint - and of course any such set has $\ell _1$ elements, and similar statements hold with $U_1$ and $U_2$ interchanged. We cannot then have both $\ell _1\ge 3$ and $\ell _2\ge 3$, because the connected union of three of the sets $f^{j}(\overline{U_1})$ and two of the sets $f^\ell (\overline{U_2})$ has three complementary components, and in whichever one another set $f^n(U_2)$ is contained, it is separated from one of the sets $f^j(\overline{U_1})$. So now we assume without loss of generality that $\ell _2=2$ and $\ell _1\ge 2$. Since $\ell _1$ and $\ell _2$ are coprime, $\ell _1$ must be odd and $\ge 3$. We have $k_1=2d$ and $A_0=A_d=\{ j+dt:0\le t<\ell _1\} $ for some fixed $j$ with $0\le j<\ell _1$. The cyclic order  on $\partial U_1$ of the points $f^{j+dt}(\partial U_2)\cap \partial U_1$, for $0\le t<\ell _1$, is the reverse of the  cyclic order on $f^d(\partial U_1)$ of the points $f^{j+dt}(\partial U_2)\cap f^d(\partial U_1)$. So if the points on $\partial U_1$ in anticlockwise order  are $f^{j+d\sigma(t)}(\partial U_2)\cap \partial U_1$ for $0\le t< \ell _1$ then  the  points on $f^d(\partial U_1)$ in  anticlockwise order are $f^{j+d\sigma (\ell _1-1-t)}(\partial U_2)\cap f^d(\partial U_1)$. Now $f^d:\partial U_1\to f^d(\partial U_1)$ is a homeomorphism which preserves anticlockwise order. So there must be $r$ such that
$$\sigma (t)+1=\sigma (r-t)$$
where integers are interpreted $\mbox{ mod }\ell _1$. Now $\ell _1$ is odd, so for any $r$ the transformation $t\mapsto r-t\mbox{ mod }\ell _1$ has a unique fixed point. It follows that the conjugate transformation $t\mapsto t+1\mbox{ mod }\ell _1$ also has a fixed point, which is absurd. So it is not possible to have both $\ell _1>1$ and $\ell _2>1$. This completes the proof of the dichotomy 3 or 4 of the lemma.

\end{proof}

\subsection{Examples of equivalence classes of different sizes}\label{3.4}

We now consider examples of matings $s_p\Amalg s_q$ for which $\mu _p$ and $\mu _q$ are not in conjugate combinatorial limbs, so that there is a rational map $f$ which is topologically conjugate to $[s_p\Amalg s_q]$. We want to concentrate on examples for which the closures of Fatou components of $f$ are all disjoint. In particular, this means that we want the endpoints of $\mu _p$ and $\mu _q^{-1}$ to have disjoint orbits, and we want the $\sim _{p,q}$-equivalence class of $\mu _p$  not to include the closure of any finite-sided gap of $L_p\cup L_q^{-1}$, and similarly for $\mu _q$.  Nevertheless, we are interested in finding examples where the $\sim _{p,q}$-equivalence class of $\mu _p$ (or $\mu _q^{-1}$) has various different sizes.

\noindent{\em{Example 1.}} Let $p=\frac{3}{7}$ and $q=\frac{3}{31}$. The minimal leaves $\mu _{1/3}$  and $\mu _{1/15}$ satisfy $\mu _{1/3}<\mu _{3/7}$ and $\mu _{1/15}<\mu _{3/31}$. By the Centrally Enlarging Lemma II.5.1 in \cite{T}, any periodic leaf in $L_{3/7}$ cannot intersect the set 
	$$\{ e^{2\pi ix}:x\in (-\textstyle{\frac{1}{7},\frac{1}{7})\cup (\frac{3}{7},\frac{4}{7}})\} $$ 
Since $\mu _{4/31}=\mu _{3/31}$ and  $(-\frac{4}{31},-\frac{3}{31})\subset (-\frac{1}{7},\frac{1}{7})$, it follows that $(\mu _{3/31})^{-1}$ is a whole equivalence class for $\sim _{3/7,3/31}$, and that the same is true for  any periodic leaf of $L_{3/31}^{-1}$ with endpoints $e^{2\pi ix_j}$ for $j=1$, $2$, and with $x_j\in (-\frac{1}{7},\frac{1}{7})$ for $j=1$, $2$. Let $\Delta $ denote the closure of the gap of $L_{3/31}$ with vertices at $e^{2\pi ix_j}$ for $x_j=2^j/15$  for $0\le j\le 3$. Then  any leaf in $L_{3/31}^{-1}$ is either  a whole equivalence class for $\sim _{3/7,3/31}$ or in the backward orbit of $\Delta ^{-1}$, because the forward orbit of any leaf in $L_{3/31}$ must intersect the set $\{ e^{2\pi ix}:x\in [\frac{1}{15},\frac{3}{31}]\cup [\frac{4}{31},\frac{2}{15}]\} $, because the image of the longest leaf in the closure of the forward orbit must be in this set, again by II.5.1 of \cite{T}. This means that there are no common endpoints between any leaves of $L_{3/7}$ and of $L_{3/31}^{-1}$. There are no finite-sided gaps for $L_{3/7}$. So every nontrivial equivalence class for $\sim _{3/7,3/31}$ is either a single leaf of $L_{3/7}$ or of $L_{3/31}^{-1}$ or in the backward orbit of $\Delta ^{-1}$. 

It follows that for the rational map $f$  (whose existence is given by Theorem \ref{3.2}) which is topologically conjugate to $[s_{3/7}\Amalg s_{3/31}]$, the closures of Fatou components are disjoint.  Let $v_1(f)$ and $v_2(f)$ be the critical points of $f$ which are the images of $\infty $ and $0$ respectively under the conjugacy $\varphi $  with $\varphi \circ [s_{3/7}\Amalg s_{3/31}]=f\circ \varphi $. We have $f\in V_{5,0}$. Theorem \ref{3.5} below can be applied to this $f$.

\medskip \medskip

\noindent{\em{Example 2}} Example 1 can be generalised.Once again let $p=\frac{3}{7}$ and let $q$ be any odd denominator  rational in $(-1/7,1/7)$ or in $(2/7,1/3)\cup (2/3,5/7)$ such that $\mu _q$ is not  in the boundary of two different gaps of $L_q$ (that is, such that $s_q$ is {\em{primitive}})  and such that the minor leaf $\mu _q$ is bounded from $0$ by a nonperiodic  leaf of $L_{3/7}$ in the backward orbit of $\mu _{3/7}$. If $q$ satisfies this condition then so does $1-q$. An explicit example is given by $q=39/127$. We have $\mu _{39/127}=\mu_{40/127}$, and this is bounded from $0$ by the leaf of $L_{3/7}$ with endpoints $e^{2\pi i(9/28)}$ and $e^{2\pi i(17/56)}$. 

\medskip \medskip 

\noindent{\em{Example 3.}} Let $p=\frac{7}{15}$ and $q=\frac{5}{31}$. We have $\mu _{7/15}=\mu _{8/15}$ and $\mu _{5/31}=\mu _{6/31}$. The minimal leaves $\mu _{1/3}$ and $\mu _{1/7}$ satisfy $\mu _{1/3}<\mu _{7/15}$ and $\mu _{1/7}<\mu _{5/31}$. 
 There are no leaves of period four in $L_{5/31}$, because there are none between $\mu _{1/7}$ and $\mu _{5/31}$ (applying II.5.1 of \cite{T}, once again). However, there are some nontrivial $\sim _{7/15,5/31}$-equivalence classes in $L_{7/15}\cup L_{5/31}^{-1}$ which are larger than the closure of a single leaf or closure of a single finite-sided gap. For example, since $\mu _{14/31}<\mu _{7/15}$ and $2^2\cdot 14=-6\mbox{ mod }31$, there is a leaf $\ell =\mu _{14/31}^4\in L_{7/15}$ with endpoints $e^{2\pi i(\pm 6/31)}$, which has an endpoint in common with $\mu _{5/31}^{-1}=\mu _{6/31}^{-1}$.  Hence $\ell \cup \mu _{5/31}^{-1}$ is in a single equivalence class. But the only period five leaves in $L_{5/31}$ are the ones in the periodic orbit of $\mu _{5/31}=\mu _{6/31}$ (again by using II.5.1 of \cite{T} ---- because there are no period five leaves separating $\mu _{1/7}$ and $\mu _{5/31}$). So $\ell \cup \mu _{5/31}^{-1}$ is the whole equivalence class. 

Note that there is one equivalence class which contains three leaves of $L_{7/15}$ and three leaves of $L_{5/31}^{-1}$, because $\mu _{3/7}\in L_{7/15}$ and $\mu _{1/7}^{-1}\in L_{5/31}^{-1}$ and the orbits of these two leaves lie in a single equivalence class. Close to this equivalence class there will be other equivalence classes which contain one leaf of $L_{5/31}^{-1}$ and two leaves of $L_{7/15}$. There are also some other equivalence classes that contain two leaves of $L_{7/15}$ at opposite ends of a leaf of $L_{5/31}^{-1}$. For example, $\mu _{7/15}$ and $\mu _{2/5}$ are joined by a leaf in $L_{5/31}^{-1}$.

Let $f$ denote the rational map which is topologically conjugate to $[s_{7/15}\Amalg s_{5/31}]$. Then we claim that the closures of the Fatou components of $f$  are disjoint. Let $U_1$ and $U_2$ denote the periodic Fatou components of $f$ containing the critical values, of periods 4 and 5 respectively. By \ref{3.3}, any point in $\partial f^{j_1}(U_1)\cap \partial f^{j_2}(U_2)$ would have to be fixed by $f$. But the only fixed equivalence classes are those of  $\mu _{1/3}$ and $\mu _{1/7}^{-1}$. The leaf $\mu _{1/3}$ is an entire equivalence class, and the equivalence class containing $\mu _{1/7}^{-1}$ is the closure of the finite-sided gap whose boundary is the periodic orbit of $\mu _{1/7}^{-1}$, which is of period three. We have $\partial U_i\cap \partial f^j(U_i)=\emptyset $ for $i=1$, $2$ and $f^j(U_i)\ne U_i$, because the only possible intersection for $i=1$ is in $\varphi (\mu _{7/15})$, where $\varphi :\overline{\mathbb C}\to \overline{\mathbb C}$ satisfies $\varphi \circ (s_{7/15}\Amalg s_{5/31})=f\circ \varphi $. But we have seen that $\varphi ^{-1}(\varphi (\mu _{7/15}))=\mu _{7/15}$. Similarly the only possible intersection for $i=2$ is in $\varphi ^{-1}(\varphi (\mu _{5/31}^{-1}))$, but we have seen that $\varphi ^{-1}(\varphi (\mu _{5/31}^{-1}))$ is the union of $\mu _{5/31}^{-1}$ and one period five leaf in $L_{7/15}$. So once again the intersection is empty. Theorem \ref{2.4} applies to this example, but Theorem \ref{3.5} will not apply.

\medskip \medskip 

The rest of the paper will be devoted to proving the following theorem. By Proposition \ref{3.3}, the conditions of the theorem imply that distinct Fatou components have disjoint closures. The conditions of the  theorem hold for Examples 1 and 2, but not for Example 3, although Fatou components are disjoint in Example 3. The description of all nearby hyperbolic components, provided by the results of Section \ref{2}, suggests that not all nearby hyperbolic components are represented by matings in Example 3.  But it is difficult to be sure, because the theory so far developed is not sufficient to exclude all matings of polynomials with two cycles of some fixed periods. 

\begin{theorem}\label{3.5} Let $f\in V_{k,0}$ with $f\simeq s_{p}\Amalg s_{q}$, where:
\begin{itemize}
\item the semiconjugacy $\varphi $ with $\varphi \circ (s_p\Amalg s_q)=f\circ \varphi $ satisfies $\varphi (\infty )=c_1(f)$;
\item each $\sim _{p,q}$-equivalence class intersecting the boundary of the minor gap of $L_p$ is contained in the boundary of that minor gap, that is, is either a point or a single leaf of $L_p$.
\item  Each leaf in the forward orbit of $\mu _q^{-1}$ is in a separate equivalence class, and disjoint from the closure of the minor gap of $L_p$.
\end{itemize} 
 Then the post-critically-finite centre of each  type IV hyperbolic component in $V_{k,0}$, in a sufficiently small neighbourhood of  the closure of the hyperbolic component of $f$,  is Thurston equivalent  $s_{y}\Amalg s_{q}$, for some  $y$ with $y$ near the minor gap of  $L_p$.\end{theorem}

\section{Proof of the first step in the induction}\label{4}

\subsection{}\label{4.1}
Here is the start of the proof of \ref{3.5}. By Proposition \ref{3.3}, the closures of Fatou components are disjoint. Fix a graph $G(f)=G$ and  Markov partition ${\mathcal{P}}(f)={\mathcal{P}}$  as in Section  \ref{2} (and Theorem 1.2 of \cite{R1}), and with the properties provided there. Let ${\mathcal{P}}_j={\mathcal{P}}_j(f)$ be the sequence of partitions constructed there, and let $P^j=P^j(f)=P^j(f,v_2)\in {\mathcal{P}}_j(f)$ with $v_2(f)\in P^j$,  as in \ref{2.4} and \ref{2.6}. Let $g\in V_{k,0}$ be postcritically finite of type IV near the hyperbolic component of $f$. 
Define $Q_1'$ to be the  union of  two intervals of $S^1$ with endpoints in  $\varphi ^{-1}(P^0(f))$ containing $\varphi ^{-1}(P^0(f))$, and such that each interval contains an endpoint of the minor leaf of $L_p$. This is possible, by the hypotheses given. As in section \ref{3}, let $s:S^1\to S^1$ be given by $s(z)=z^2$.  Let $T_m$ be the two-valued  local inverse of $s^m$ defined on $Q_1'$ with $T_m(\varphi ^{-1}(F_2(f)))=\varphi ^{-1}(F_2(f))$. Then $T_m(Q_1')\subset Q_1'$. For some integer $i_1$, $T_m^{i_1}(Q_1')$ is in a sufficiently small neighbourhood of $\varphi ^{-1}(P^{i_1m}(f))$ that $T_m^{i_1}(Q_1')\subset \varphi ^{-1}(P^0(f))$.   We define $r_1=i_1m$. From now on we assume that $j_1$ is large enough that $j_1>3i_1$. 
 
Then, as in \ref{2.4} and \ref{2.6}, we define $j_1$ to be the largest integer $j$ such that $v_2(g)\in  P^{jm}(g,f)$, that is, the largest integer $j$ such that $P^{jm}(g,f)=P^{jm}(g,v_2)=P^{jm}(g)$. Again as in \ref{2.4} and \ref{2.6}, we define $n_1=j_1m$. We define $Q_1=T_m^{j_1}(Q_1')$ and then 
$$\varphi ^{-1}(P^{n_1}(f))\cap S^1 \subset Q_1\subset \varphi ^{-1}(P^{(j_1-i_1)m}(f))=\varphi ^{-1}(P^{n_1-r_1}(f)).$$

As in \ref{2.6}, we define $m_2$ to be the least integer $>n_1$ with 
$$g^{m_2}(v_2(g))\in P^{n_1-r_1}(g).$$
  We then have a local inverse $S_{m_2,g}$ defined on $P^{n_1-r_1}(g)$ mapping $g^{m_2}(v_2(g))$ to $v_2(g)$, and a corresponding two-valued  local inverse $S_{m_2,h}$ of $h^{m_2}$ for all $h\in V(P^0,\cdots P^{n_1};f)$ including $h=f$. Since $v_2(g)\in P^{n_1}(g)$ we have $m_2>n_1$ and since $j_1>3i_1$ we have 
$$S_{m_2,h}(P^{n_1-r_1}(h))\subset P^{n_1}(h).$$
We have
$$g^j(S_{m_2,g}(P^{n_1-r_1}(g)))\cap P^{n_1-r_1}(g)=\emptyset,\ \ \ 0<j<m_2,$$
because $g^j(v_2)\notin P^{n_1-r_1}(g)$ for $0<j<m_2$, and hence also 
$$f^j(S_{m_2,f}(P^{n_1-r_1}(f)))\cap P^{n_1-r_1}(f)=\emptyset,\ \ \ 0<j<m_2.$$
  Now $S_{m_2,f}$ is two-valued on $P^{n_1-r_1}(f)$ and also on the connected set \\ $S_{m_2,f}(P^{n_1-r_1}(f))$ --- which does not contain $v_2(f)$, since 
  $$v_2(f)\in P^{n_1+m}(f)=S_{m,f}(P^{n_1}(f)).$$
   Therefore, there are two single-valued local inverses $S_{m_2,i,f}$ of $f^{m_2}$ defined on $S_{m_2,f}(P^{n_1-r_1}(f))$, with image in $S_{m_2,f}(P^{n_1}(f))$, for $i=1$, $2$, and each of the sets 
$$\bigcap _{n>0}S_{m_2,i,f}^nS_{m_2,f}(P^{n_1-r_1}(f))=\bigcap _{n>0}S_{m_2,i,f}^nS_{m_2,f}(P^{n_1}(f))$$
contains a single point of period $m_2$, whose periodic orbits intersect $P^{n_1}(f)$ just in these points. We call these points $z_{2,1}(f)$ and $z_{2,2}(f)$, or $z_{2,1}$ and $z_{2,2}$, if no confusion can arise. 

\begin{lemma}\label{3.6}  For at least one of $i=1$ or $2$ the points in $\varphi ^{-1}(z_{2,i})\cap S^1$ are of period $m_2$ under $s$\end{lemma}
\noindent{\textbf{Remark}} Each set $\varphi ^{-1}(z_{2,i})\cap S^1$ contains either one or two points, but we do not use this fact at this stage: and the corresponding fact in the general inductive step might not be true, under the hypotheses we are using.
\begin{proof} We write $T_{m_2}$ for the two-valued  local inverse defined on $\varphi ^{-1}(P^{n_1-r_1}(f))$, with 
$$\varphi \circ T_{m_2}=S_{m_2,f}\circ \varphi .$$
In particular, $T_{m_2}$ is defined on $Q_1$, and since $\varphi (Q_1)$ is connected, and the image under $\varphi $ of each component of $Q_1$ intersects $P^{n_1}(f)$, we have  $T_{m_2}(Q_1)\subset Q_1$. Since $T_{m_2}$ is a two-valued contraction, there are exactly two points in $Q_1$ which are fixed under $T_{m_2}$, and thus have period $m_2$ under $s$. The images under $\varphi $ of these points are fixed under $f^{m_2}$ and hence must be the points $z_{2,i}$: either mapped to the same point $z_{2,i}$ or one to each of $z_{2,1}$ and$z_{2,2}$. So for at least one of $i=1$ or $2$, $\varphi ^{-1}(z_{2,i})\cap S^1$ consists of one or two points of period $m_2$ under $s$ .\end{proof}

\subsection{} \label{3.7}
From now on we fix $z_2(f)=z_{2,i}(f)$ for one of $i=1$, $2$ so that $\varphi ^{-1}(z_2(f))\cap S^1\subset Q_1$ consists of one or two points of period $m_2$. We take one of these points to be $e^{2\pi iy_2}$.  Choose a point $e^{2\pi iw_2}$ in the boundary of the minor gap of $L_p$, and in the same component $Q_{1,1}$ of $Q_1$ as $e^{2\pi iy_2}$,  and of period $mu_2$ for the least $u_2$ such that $T_m^{u_2}(Q_{1,1})\subset Q_{1,1}$. This point $w_2$ does exist, because each component of $Q_1'$ intersects the boundary of the minor gap of $L_p$, and since $Q_1=T_m^{j_1}(Q_1')$, the same is true for $Q_1$. Moreover, $u_2\le j_1+1$, because each interval of $Q_1'$ contains a preimage under $s$ of the other and $Q_{1,1}$ is a preimage under $s^{j_1m}$ of one of the intervals of $Q_1'$. Thus $s_{w_2}\Amalg s_q$ is equivalent to a rational map $h$, which is a tuning of $f_1$, and in the Mandelbrot set copy containing $f_1$. We write $f_{3/2}$ for the parabolic map on the boundary of the hyperbolic component containing $h$, where the parabolic point is fixed by $f_{3/2}^{km}$, if $km$ is the period of $e^{2\pi iw_2}$ under $s$. There is a continuous map $\varphi _{3/2}$ such that $\varphi _{3/2}\circ (s_p\Amalg s_q)=f_{3/2}\circ \varphi _{3/2}$ on $S^1\cup L_p\cup L_q^{-1}$. 

For $h\in V(P^0\cdots P^{n_1-r_1},f)$,there are exactly two  points of period $m_2$ in $P^{n_1-r_1}(h)$, except for  the map$h_2\in V(P^0\cdots P^{n_1},f)$ with a single parabolic periodic point of period $m_2$ and multiplier $m_2$. On any path in $V(P^0\cdots P^{n_1-r_1},f)$ starting from $f$ which avoids $h_2$, the function $z_2(h)$ can be chosen to be single-valued and continuous taking the previously defined value $z_2(f)$ at $f$. Two paths to $h$ which combine to give a simple closed loop round $h_2$ give different values to $z_2(h)$. So  we have  a path $f_t$ for $t\in [3/2,t_{j_1-i_1}]$ with $t_{j_1-i_1}>2$ such that 
$$v_2(f_2))=z_2(f_2),$$
$$f_t\in V(P^0,\cdots P^{n_1},f)=V(P^0,\cdots P^{n_1},f_{3/2})\mbox{ for all }t\in [3/2,t_{j_1}],$$
 and we can choose $\varepsilon $ so that 
\begin{equation}\label{3.5.0}f_{3/2+\varepsilon }\in V(P^0,\cdots P^{im};f_{3/2})\end{equation}
 for $i$ which can be taken arbitrarily large by taking $\varepsilon $ arbitrarily close to $0$. We can choose $f_t$ for $t\in [3/2,3/2+\varepsilon ]$ also, so that (\ref{3.5.0}) holds also for $t\in [3/2,3/2+\varepsilon ]$ replacing $3/2+\varepsilon $ by $t$. We can also choose $f_t$ so that 
 $$[3/2+\varepsilon,2]=\bigcup _{j=i-1}^{j_1}[t_{j+1},t_j],\ \ [3/2,3/2+\varepsilon ]=[t_{i+1},t_i],$$
 and 
\begin{equation}\label{3.5.4}f_t\subset V(P^0,\cdots P^{jm},f_{3/2})\setminus V(P^0,\cdots P^{(j+1)m},f_{3/2}),\ \ t\in [t_j,t_{j+1}],\ \ j_1\le j\le i.\end{equation}
Of course this implies that 
$$f_{t_j}\in \partial V(P^0,\cdots P^{jm},f_{3/2}),\ \ j_1-i_1\le j\le i.$$

We have a homeomorphism $\psi _t$, $t\in [3/2,t_{j_1}]=[3/2,2]$,  with $\psi _{3/2}=\mbox{identity}$ which maps $f_t^i(S_{m_2,f_t}(P^{jm}(f_t)))$  to $f_{3/2}^i(S_{m_2,f_{3/2}}(P^{jm}(f_{3/2})))$ for $0\le i<m_2$,  and $f_t^i(z_2(t))$ to $f_{3/2}^i(z_2(3/2))$ for $0\le i<m_2$. We can also choose $\psi _t $ so  that
\begin{equation}\label{3.5.1}f_{3/2}\circ \psi _t=\psi _t\circ f_t\mbox{ on }f_t^{-1}(P^{n_1-i_1m}(f_t)), t\in [3/2,2].\end{equation}
Note that $\psi_t$ does not map $v_2(f_t)$ to $v_2(f_{3/2})$. In fact, the conditions imposed imply that
$$\beta (t)=\psi _t(v_2(f_t)), \ \ t\in [3/2,2]$$
 is a path from $v_2(f_{3/2})$ to $z_2(3/2)\in P^{n_1}(f_{3/2})$ and 
 $$\zeta (t)=\psi _t(c_2(f_t))$$
 is a path from $c_2(f_{3/2})$ to the periodic preimage of of $z_2(3/2)$ satisfying $f_{3/2}\circ \zeta =\beta $, by (\ref{3.5.1}). We also write 
$$\beta _j=\beta \mid [t_{j+1},t_j],\ j_1\le j\le i,$$

Then we have 
\begin{equation}\label{3.5.5}f_{3/2}(\zeta (t))=f_{3/2}(\psi _t(c_2(f_t)))=\psi _t(f_t(c_2(f_t))=\psi _t(v_2(f_t))=\beta (t),\ \ t\in [3/2,t_{j_1}].\end{equation}
In particular, 
\begin{equation}\label{3.5.6} f_{3/2}\circ \beta =\zeta .\end{equation}

\begin{lemma}\label{3.30} We can choose $\beta _j$ for $j\le j_1$ so that, for each $0<i\le i_1$, either $\beta $ and $f^{im}(\beta )$ are disjoint, or $i$ is a multiple of $u_2$ and $\beta \subset f^{im}(\beta )$.\end{lemma}
\begin{proof} Each set $\partial V(P^0,\cdots P^{jm};f_{3/2})$ is a topological circle containing the point $\beta (t_j)=\beta _j(t_j)$. This follows from the Theorem 3.2 of \cite{R1}: points in $\partial V(P^0,\cdots P^{jm},f_{3/2})$ are in one-to-one correspondence with $\partial P^{jm}(f_{3/2})$, via the map $h\mapsto \psi _h(v_2(h))$, where 
$$\psi _h\circ h=f_{3/2}\circ h\mbox{ on }\bigcup _{\ell \le (j-j_1+i_1)m}h^\ell \partial P^{jm}(h),$$
 and this map is seen to be continuous by looking at the sets $V(\underline{Q})$ for infinite sequences $\underline{Q}$ with 
$$V(\underline{Q})\subset \partial V(P^0,\cdots P^{jm},f_{3/2})=V(P^0,\cdots P^{(j-1)m},\partial P^{jm},f_{3/2}).$$
Each of the sets $V(\underline{Q})$ is a point, by Theorem 3.2 of \cite{R1}.   Each set $P^{jm}(f_{3/2})\setminus P^{(j+1)m}(f_{3/2})$ for $j>j_1$ is disjoint from the forward orbit of $z_2(3/2)$ and any two paths in an annulus joining distinct boundary components are homotopic via twists round those boundary components. Also any two  paths in an annulus joining the inner boundary component to a fixed point in the interior of the annulus are homotopic via a homotopy fixing the interior endpoint and moving the other endpoint on the annulus boundary.  
Let $S_{mu_2}$ be the single-valued  local inverse of $f_{3/2}^{mu_2}$ which fixes $\varphi _{3/2}(e^{2\pi iw_2})$ and write
$$S_{mu_2}=S_{m,u_2}\circ \cdots \circ S_{m,1},$$
where each $S_{m,\ell }$ is a local inverse of $f_{3/2}^m$. We need to define the $\beta _j=\beta _{j,0}$ inductively, in terms of a larger sequence $\beta _{j,i}$ for $0\le i<u_2$, and $j\le j_1-i_1+i$ if $i\le i_1'=\mbox{min}(i_1,u_2)$ and, if $u_2>i_1$, then also for $j\le j_1+i-i_1+u_2$ if $i_1<i<u_2$. 
 Writing $\beta _{j,i}=\beta _{j,i'}$ if $i-i'=0\mbox{ mod }u_2$, we will have
\begin{equation}\label{3.30.2}\beta _{j_1-i,i}=f_{3/2}^{mi}(\beta _{j_1}),\ 0\le i\le i_1\end{equation} 
so that, if $u_2\le i_1$, then $\beta _j$ is defined for $j\ge j_1-nu_2$ for the largest integer $n$ with $nu_2\le i_1$. More generally and  precisely, we will have
\begin{equation}\label{3.30.1}\beta _{j,i}=S_{m,i}\beta _{j-1,i+1}.\end{equation}
with $i+1=u_2$ replaced by $0$ if $i=u_2-1$, wherever both $\beta _{j,i}$ and $\beta _{j-1,i+1}$ are defined. 

We start by choosing the $\beta _{j_1-i,i}$ to be arcs  satisfying (\ref{3.30.2}) for $i_1\ge i\ge 0$. Since
$$\beta _{j_1-i,i}\subset P^{(j_1-i)m}(f_{3/2})\setminus S_{m,f_{3/2}}(P^{(j_1-i)m}(f_{3/2})),$$
these arcs are disjoint, apart from possibly a common endpoint between $\beta _{j_1-i,i}$ and $\beta _{j_1-i-1,i+1}$, which is bound to hold if $u_2=1$, and can be avoided by simply moving the second endpoint of $\beta _{j_1}$ otherwise. Then for $0\le \ell \le i_1'$,  choose the arcs $\beta _{j_1-i_1'+\ell,i}$ to be disjoint for $i_1'\ge i\ge i_1'-\ell $ by doing this successively for $0\le \ell \le i_1'$ , by moving the second endpoint of $\beta _{j_1-i_1'+\ell,i_1'}$. If $i_1'=u_2-1$ then we define all $\beta _{j,i}$ for $j>j_1-i_1+i_1'$ using equation (\ref{3.30.1}) and all arcs are disjoint apart from some common endpoints. If $i_1'=u_2-1$, then the proof is finished. If $i_1'<u_2-1$, that is, $i_1'=i_1<u_2-1$, then we continue to choose the $\beta _{j_1-i_1+\ell,i_1}$ so that the arcs $\beta _{j_1-i_1+\ell,i}$ are disjoint for $i_1\ge i\ge i_1-\ell $ and $i_1< \ell \le u_2-1$. In particular, this is true for $\ell =u_2-1$.  Then once this is done, we choose the remaining $\beta _{j,i}$ so that (\ref{3.30.1}) is satisfied, and this implies that, for each $j\ge j_1-i_1+u_2-1$, the arcs $\beta _{j,i}$ are disjoint for $0\le i<u_2$, which is what is required.

\end{proof}

   If $\gamma :[a,b]\to \overline{\mathbb C}$ is an arc,  then we define $\sigma _\gamma $ to be a homeomorphism which is the identity outside a suitably small disc $D_\gamma $ neighbourhood of $\gamma $ and satisfies $\sigma _\gamma (\gamma (a))=\gamma (b)$. 

  \begin{figure}
\centering{\includegraphics[width=4cm]{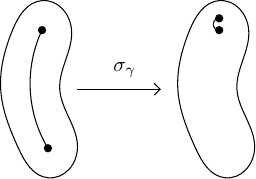}}
\caption{$\sigma _{\gamma }$}
\end{figure}

This does not define $\sigma _\gamma $ uniquely but does define it up to isotopy preserving $A$, if $A\subset \overline{\mathbb C}$ is any set disjoint from $D_\gamma $, that is, disjoint from $\gamma $, and $D_\gamma $ is sufficiently small. Similarly, if  $\gamma =\gamma _1*\cdots *\gamma _r$ is any path which can be written as a union of arcs $\gamma _i$, then we also define
 $$\sigma _\gamma =\sigma _{\gamma _r}\circ \cdots \circ \sigma _{\gamma _1}.$$
 
 \begin{lemma} \label{3.8}
 \begin{equation}\label{3.8.1}f_2\simeq \sigma _{\zeta }^{-1}\circ \sigma _\beta \circ f_{3/2}.\end{equation}
 \end{lemma}
 \begin{proof} Define $\beta ^t=\beta \vert [t,2]$ and $\zeta ^t=\zeta \vert [t,2]$. Then
 $$\sigma _{\zeta ^t}^{-1}\circ \sigma _{\beta ^t}\circ \psi _t\circ f_t\circ \psi _t^{-1}$$
 are homotopic postcritically-finite branched coverings, for $t\in [3/2,2]$.  The paths $\beta ^2$ and $\zeta ^2$ are trivial, while $\beta ^{3/2}=\beta $ and $\zeta ^{3/2}=\zeta $, and $\psi _{3/2}$ is the identity. 
 \end{proof}
  
  Now we are ready to complete the proof of the first step in the induction, with the following lemma. 
 
 \begin{lemma}\label{3.9}
 \begin{equation}\label{3.5.3}s_{y_2}\Amalg s_q\simeq f_2.\end{equation}
 \end{lemma}
 \begin{proof}
 Let $\alpha $ be the union of a path in the minor gap of $L_p$ from the critical value of $s_p\Amalg s_q$ to $e^{2\pi iw_2}$ and an arc  in $Q_1$ (that is, an arc of $S^1$) from $e^{2\pi iw_2}$ to $e^{2\pi iy_2}$.  Let $\omega $ be a path from the critical point to the periodic preimage of $e^{2\pi iy_2}$ satisfying $(s_p\Amalg s_q)(\omega )=\alpha $. The path $\alpha $ passes through no point in the forward orbit of $e^{2\pi iy_2}$ before its endpoint, and similarly for the path $\omega $. It follows immediately that 
 \begin{equation}\label{3.9.1}s_{y_2}\Amalg s_q\simeq \sigma _\omega ^{-1}\circ \sigma _\alpha \circ (s_p\Amalg s_q).\end{equation}
 We claim that  $\varphi _{3/2}\circ \alpha $ and $\beta $ are homotopic via a homotopy preserving endpoints and the forward orbit of $z_2(3/2)=\varphi _{3/2}(e^{2\pi iy_2})$, assuming that $\varphi _{3/2}$ is modified in the Fatou components to map critical points to critical points and critical values to critical values, and hence $\varphi _{3/2}\circ \omega $ and $\zeta $ are similarly homotopic. We see this as follows. We replace $\varphi _{3/2}$ by a homeomorphism $\varphi _{3/2,0}$ which is arbitrarily close to $\varphi _{3/2}$, and which maps the forward orbit of $z_2$ and the first $i_1m$ images of the critical values in the same way  as $\varphi _{3/2}$. Both $\beta $ and $\varphi _{3/2,0}(\alpha )$ are paths in $P^{n_1-r_1}(f_{3/2})$, with the same endpoints. In fact, $\beta $ is a path in $P^{n_1}(f_{3/2})$. By \ref{3.30}, there are no transversal intersections between $\beta $ and $f^{im}(\beta )$ for $0\le i\le i_1$ up to homotopy preserving the forward orbit of $z_2(3/2)$.    There are also no transversal intersections between $\alpha $ and $s^{im}(\alpha )$ for $0\le i\le i_1$, and hence also no transversal intersections between $\varphi _{3/2,0}(\alpha )$ and $f_{3/2}^{im}(\varphi _{3/2,0}(\alpha ))$ for $0\le i\le i_1$. By the definition of $m_2$, the  points in the forward orbit of $z_2(3/2)$ which lie in $P^{n_1-r_1}(f_{3/2})$ are precisely the points $f_{3/2}^{im}(z_2(3/2))$ for $0\le i\le i_1$, recalling that $r_1=i_1m$. So we need to see that the two paths $\varphi _{3/2,0}(\alpha )$ and $\beta $ are homotopic via a homotopy which does not pass through these points in the forward orbit. This follows from the invariance property. Suppose that $\beta $ and $\varphi _{3/2,0}(\alpha )$ are not homotopic in this way. Extend $\alpha $ to  path $\alpha '$ ending in $\partial Q_1$. The path $\alpha '$ still has the property that there are no transversal intersections between $\alpha '$ and $s^{im}(\alpha ')$ for $0\le i\le i_1m$. Then $\beta $ must have transverse intersections with $\varphi _{3/2,0}(\alpha ')$ which cannot be removed by homotopy. We concentrate on the final  essential intersection, on $\beta $, of $\beta $ with $\varphi _{3/2,0}(\alpha ')$, before the endpoint of $\beta $ at $z_2(3/2)$. This arc must bound an arc on $\varphi _{3/2,0}(\alpha )$ containing some point $f_{3/2}^{im}(v_2)$ for some $0<i\le i_1$. The image under $f_{3/2}^{im}$ of this final arc is bound to intersect itself. In the picture, $\varphi _{3/2,0}(\alpha ')$ is drawn as a straight line. 
 
  \begin{figure}
\centering{\includegraphics[width=8cm]{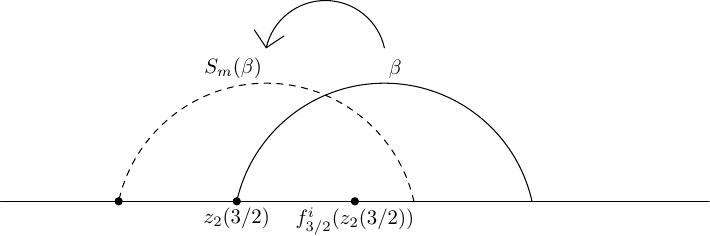}}
\caption{The final arc on $\beta $ and intersection with $S_m\beta $}
\end{figure}

 It follows that $\beta $ and $f_{3/2}^{im}(\beta )$ have transverse intersections, which gives the required contradiction to \ref{3.30}. 
 
 Since $\varphi _{3/2,0}$ can be chosen arbitrarily close to $\varphi _{3/2}$, we deduce that 
 $$ \sigma _\omega ^{-1}\circ \sigma _\alpha \circ (s_p\Amalg s_q)\simeq \sigma _{\zeta }^{-1}\circ \sigma _\beta \circ f_{3/2},$$
 and hence we have \ref{3.5.3}.
 \end{proof}
 
This completes the first step in showing that $f_2$ is equivalent to a mating with $s_q$. The equivalence provides a continuous surjective map $\varphi _2$, approximated arbitrarily closely by homeomorphisms, with 
$$\varphi _2\circ (s_{y_2}\Amalg s_q)=f_2\circ \varphi _2.$$
In order to continue, we need the following. 

\begin{lemma} \label{3.40} For any $n\ge 0$, let $R(g)$ be an isotopically varying set of ${\mathcal{P}}_n(g)$ which is not strictly contained in the backward orbit of $P^{n_1-r_1}(g)$ for $g\in V(P^0,\cdots P^{n_1},f_1)$. Then
 $$\varphi _2^{-1}(R(f_2))=\varphi _{3/2}^{-1}(R(f_{3/2}))=\varphi _1^{-1}(R(f_1)). $$

   \end{lemma}
   \begin{proof} This uses \ref{3.8} and \ref{3.9}. (\ref{3.8.1}) is an equivalence between $f_2$ and $\sigma _\zeta ^{-1}\circ \sigma _\beta  \circ f_{3/2}$. (\ref{3.9.1}) is an equivalence between  $s_{y_2}\Amalg s_q$ and 
   $\sigma _\omega ^{-1}\circ \sigma _\alpha \circ (s_p\Amalg s_q)$. The main part of \ref{3.9} is to replace (\ref{3.8.1}) by an equivalence between $f_2$ and $\sigma _{\varphi _{3/2}(\omega )}^{-1}\circ\sigma _{\varphi _{3/2}(\alpha )}\circ f_{3/2}$. A homeomorphism  $\chi_{1,0}$ can be chosen  with
   $$\sigma _{\varphi _{3/2}(\omega )} ^{-1}\circ \sigma _{\varphi _{3/2}(\alpha )} \circ (s_p\Amalg s_q)\simeq _{\chi _{1,0}}f_2$$
   and to  map each set $Q(f_{3/2})$ to the corresponding set $Q(f_2)$ for $Q(h)\in {\mathcal{P}}^{n_1-r_1}(h)$.  By (\ref{3.9.1}), a homeomorphism $\chi _{2,0}$ can chosen with
   $$s_{y_2}\Amalg s_q\simeq _{\chi _{2,0}}\sigma _\omega ^{-1}\circ \sigma _\alpha \circ (s_p\Amalg s_q)$$
   where $\chi _{2,0}$ is the identity outside $\varphi _{3/2}^{-1}(P^{(j_1-i_1)m}(f_{3/2}))$. Let $\varphi _{3/2,0}$ be a homeomorphism approximating $\varphi _{3/2}$, which maps the periodic orbit of $\infty $ to the periodic orbit of $c_2$, the points $0=c_1(s_p\Amalg s_q)$ and $(s_p\Amalg s_q)(0)=v_1(s_p\Amalg s_q)$ to $c_1(f_{3/2})$ and $v_1(f_{3/2})$ and the periodic orbit of $e^{2\pi iy_1}$ to the periodic orbit under $f_{3/2}$ of $\varphi _{3/2}(e^{2\pi iy_1})$, and hence, provided that $\varphi _{3/2,0}$ is a sufficiently close approximation to $\varphi _{3/2}$,
   $$\sigma _{\varphi _{3/2,0}(\omega )}^{-1}\circ \sigma _{\varphi _{3,0}(\alpha )}\circ f_{3/2}\simeq _{\varphi _{3/2,0}}\sigma _\omega ^{-1}\circ \sigma _\alpha \circ (s_p\Amalg s_q).$$
   Now define
   $$\varphi _{2,0}=\chi _{1,0}\circ\varphi _{3/2,0}\circ \chi _{2,0}$$
   and then define  the sequence of homeomorphisms $\varphi _{2,i}$ inductively by 
   $$f_2\circ  \varphi _{2,i+1}=\varphi _{2,i}\circ (s_{y_2}\Amalg s_q).$$
   and $\varphi _{2,i}$ maps the periodic orbits of $0$ and $\infty $ under $s_{y_2}\Amalg s_q$ to the periodic orbits of the critical points of $f_2$. Then
   $$\varphi _{2,j}=\chi _{1,j}\circ \varphi _{3/2,j}\circ \chi _{2,j}$$
   where the sequences $\chi _{1,j}$, $\chi _{2,j}$ and $\varphi _{3/2,j}$ are defined inductively by:
   $$f_2\circ \chi _{1,j+1}=\chi _{1,j}\circ (\sigma _{\varphi _{3/2,j}(\omega )}^{-1} \circ \sigma _{\varphi _{3/2,j}(\alpha )} \circ f_{3/2})$$
   and $\chi _{1,j+1}$ maps $Q(f_{3/2})$ to $Q(f_2)$ for all $Q(f_{3/2})\in {\mathcal{P}}_{n_1}(f_{3/2})$;
   $$(\sigma _{\varphi _{3/2,j}(\omega )}^{-1}\circ \sigma _{\varphi _{3/2,j}(\alpha )} \circ f_{3/2})\circ \varphi _{3/2,j+1}=\varphi _{3/2,j}\circ (\sigma _\omega ^{-1}\circ \sigma _\alpha \circ(s_p\Amalg s_q));$$
   and $\varphi _{3/2,j+1}$ maps the periodic orbit of $e^{2\pi iy_2}$ to the periodic orbit of $\varphi _{3/2,j}(e^{2\pi iy_2})$, and finally
   $$(\sigma _\omega ^{-1}\circ \sigma _\alpha \circ(s_p\Amalg s_q))\circ \chi _{2,j+1}=\chi _{2,j}\circ s_{y_2}\Amalg s_q,$$
   and $\chi _{1,j}$ fixes the periodic orbit of $e^{2\pi iy_2}$.
   
   The expanding properties of $f_2$ and $f_{3/2}$ ensure that $\varphi _{2,j}$ converges to $\varphi _2$ and $\varphi _{3/2,j}$ converges to $\varphi _{3/2}$ outside the backward orbit of $\varphi _{3/2}^{-1}(P^{n_1-r_1}(f_{3/2}))$ and $\chi _{1,j}$ converges to $\chi _1$ outside the backward orbit of $P^{n_1-r_1}(f_{3/2})$.
   
The inductive definitions and the definitions of the paths  $\alpha $, $\omega $ ensure that  $\chi _{2,j'}$ is the identity on 
$(s_{y_2}\Amalg s_q)^{-j}(\varphi _{3/2}^{-1}(P^{n_1-r_1}(f_{3/2})))$ for $j'\le j$ and $\varphi _{3/2,j'}$ maps the sets $(s_{y_2}\Amalg s_q)^{-j}(\varphi_{3/2}^{-1}(P^{n_1-r_1}(f_{3/2})))$  to sets $f_{3/2} ^{-j}(P^{n_1-r_1}(f_{3/2}))$ and $\chi _{1,j'}$ maps sets $f_{3/2}^{-j}(P^{n_1-r_1}(f_{3/2}))$ to sets $f_2^{-j}(P^{n_1-r_1}(f_{2}))$. It follows that, outside the backward orbit of  $P^{n_1-r_1}(f_{3/2})$, 
$$\varphi _2=\chi _1\circ \varphi _{3/2}$$
and the result follows.
        
   \end{proof}

 \section{Second step in the induction}\label{5}
Before doing the general step in the induction, we consider the second step, that is, the construction of $f_3$. This case is a little more complicated that the general step. We need to construct the set $Q_2$, and, as forewarned in \ref{2.6}, a set $P^{n_2-r_2,1}(f_2)\subset P^{n_2-r_2}(f_2)$ which contains $\varphi _2(Q_2)$. This, in turn, is constructed from a set $Q_2'=s^{k-m+(j_2-1)m_2}(Q_2)$ for a suitable $0\le k<m$. From now on we assume, as we may do, that the sets $f^i(P^0(f))$, for $0\le i<m$, are disjoint. We also assume as we may do that $i_1$ is even, so that $r_1=i_1m$ is also even.  We define $m_2'$ to be the largest integer $i<m_2$ with $g^i(v_2(g))\in P^{r_1}(g)$. Then $m_2'\ge n_1-r_1$, and if $m_2'>n_1-r_1$, we must have $m_2'\ge n_1+m$. 

We need the following.

\begin{lemma}\label{3.10} Define $n_2-r_2=m_2'+(j_2-1)m_2$.
There exists $Q_2\subset S^1$ such that every component of $Q_2$ intersects the minor gap of $L_{y_2}$, and a set $Q_3'\subset S^1$, consisting of at most two components, both intersecting the minor leaf $\mu _{y_2}$ of $L_{y_2}$, and  there are  unions 
$$P^{n_2-r_2,0}(f_2),\ P^{n_2-r_2,1}(f_2),\ P^{n_3-r_3,0}(f_2)$$
  of sets of 
  $${\mathcal{P}}_{m_2}(f_2),\ {\mathcal{P}}_{j_2m_2}(f_2),\ {\mathcal{P}}_{m_2}(f_2),$$
   where all these unions are  closed topological discs, with
\begin{equation}\label{3.10.7}P^{n_2-r_2,0}(f_2)\subset P^{n_3-r_3,0}(f_2)\subset P^{m_2'}(f_2),\end{equation}
\begin{equation}\label{3.10.6}P^{n_2-r_2,1}(f_2)\subset P^{n_2-r_2}(f_2),\end{equation}
and  such that
 
\begin{equation}\label{3.10.5}\begin{array}{l}\varphi _2^{-1}(P^{n_2}(f_2))\cap S^1\subset Q_2\subset \varphi _2^{-1}(P^{n_2-r_2,1}(f_2))\cap S^1\\ \subset Q_3'\subset\varphi _2^{-1}(P^{n_3-r_3,0}(f_2)).\end{array}\end{equation}
Also

\begin{equation}\label{3.10.1}g^i(v_2(g))\notin P^{n_2-r_2,1}(g)=\emptyset ,\ 0<i<m_3,\ i\ne m_2,\end{equation}
and for $h\in V(P^0,\cdots P^{n_2},g)$,
\begin{equation}\label{3.10.11}h^i(P^{n_2-r_2,0}(h))\cap P^{n_2-r_2,0}(h)=\emptyset, 0<i<m_2,\end{equation}
\begin{equation}\label{3.10.3}h^i(S_{m_3,h}(P^{n_2-r_2,1}(h)))\cap S_{m_3,h}(P^{n_2-r_2,1}(h))=\emptyset , 0<i<m_3,\end{equation}
\begin{equation}\label{3.10.2}h^i(S_{m_3,h}(P^{n_2-r_2,1}(h)))\cap P^{n_2-r_2,1}(h)=\emptyset , 0<i<m_3,\ i\ne m_2,\end{equation}
\begin{equation}\label{3.10.12}h^i(P^{n_3-r_3,0}(h))\cap P^{n_3-r_3,0}(h)=\emptyset, 0<i<m_2.\end{equation}

\end{lemma}

\begin{proof}
By the definition of $m_2$, there is  a least $0<k\le m$ such that $f_2^{m_2-k}(v_2)\notin f_2^i((P^m(f_2)))$ for all $0<i\le m $ and hence such that $f_2^{m_2-k}(S_{m_2,f_2}(P^{0}(f_2)))$  is disjoint from $f_2^i(P^m(f_2))$, $0<i\le m$. We have, by \ref{3.40},
$$\varphi ^{-1}(f_1^i(P_{jm}(f_1)))=\varphi _2^{-1}(f_2^i(P^{jm}(f_2)))\mbox{ for }j\le j_1.$$
We define $Q_2'$ to be a union of at most two intervals of $S^1$, with endpoints in $\varphi _2^{-1}(f^{-1}(f^{m-k+1}(P^{m}(f_2)))\setminus f^{m-k}(P^m(f_2)))$, containing the two shorter complementary intervals of of $s_{y_2}^{-k}(\mu _{y_2})$, where $\mu _{y_2}$ is the minor leaf of $L_{y_2}$. This set is disjoint from its images under $s^i$, for $0<i\le m$. We can also choose $\varepsilon >0$ so that the $\varepsilon $-neighbourhood of $Q_2'$ is disjoint from its forward images under $s^i$ for $0<i<m$. 
Similarly $Q_2''$ is defined to be the union of two intervals in $s^{k-m_2}(Q_2')$ which intersect $\varphi _2^{-1}(P^{m_2}(f_2))$. Then we define $f_2^{m_2-k}(P^{n_2-r_2,0}(f_2))$ to be the union of sets of ${\mathcal{P}}_k(f_2)$ which intersect $\varphi _2(Q_2')$, together with any complementary components within a small neighbourhood, so that the unoin of all these sets is a closed topological disc. Then $P^{n_2-r_2,0}(f_2)$ is similarly the union of sets of ${\mathcal{P}}_{m_2}(f_2)$ which intersect $\varphi _2(Q_2'')$.  Then we define
$$P^{n_2-r_2,1}(f_2)=S_{m_2,f_2}^{(j_2-1)m_2}(P^{n_2-r_2,0}(f_2)),$$
$$Q_2=T_{m_2}^{j_2-1}Q_2'',$$
where $T_{m_2}$ is the (multivalued) local inverse of $s^{m_2}$ such that $\varphi _2\circ T_{m_2}=S_{m_2,f_2}\circ \varphi _2$
These  definitions ensure that the first two inclusions of (\ref{3.10.5}) are satisfied.

 Inverse images under $\varphi $ of sets  in the partition ${\mathcal{P}}_0(f_1)$ (and hence also ${\mathcal{P}}_k(f_1)$) can be assumed to be arbitrarily close to subsets 
of $\{ z:|z|\le 1\}$  and  $\{ z:|z|\le 1\}$. and we can also assume, replacing ${\mathcal{P}}_0(h)$ by ${\mathcal{P}}_{r_0}(h)$ for some $r_0$ depending only on $\varepsilon $ if necessary, that $\varphi _2^{-1}(f_2^{m_2-k}(P^{n_2-r_2,0}(f_2)))$ is within 
$\varepsilon $ of $Q_2'$ and hence disjoint from its forward images under $s^i$ for $0<i<m$. Now the Markov property for ${\mathcal{P}}_k(f_2)$ implies a Markov property for $Q_2'$. In particular   $s^i(Q_2)$ is either contained in $s^j(Q_2)$ or disjoint from it for $0\le i<j\le j_2m_2-k$. This Markov property and the definition of $P^{n_2-r_2,1}(f_2)$ then imply a Markov property for sets $f_2^{i}(P^{n_2-r_2,1}(f_2))$ in that  for any $0\le i<j\le j_2m_2-k$, the $f_2^i(P^{n_2-r_2,1}(f_2))$ is either contained in or disjoint from $f_2^{j}(P^{n_2-r_2,1}(f_2))$. Let $i_0$ be the number of sets of ${\mathcal{P}}_k(f_2)$ in $f_2^{j_2m_2-k}(P^{n_2-r_2,1}(f_2))=f_2^{m_2-k}(P^{n_2-r_2,0}(f_2)$. By choice of $i_1$, we can assume that $i_0\le i_1/10$. 

Let $t$ be the largest integer  with $t<m_2-k$ such that 
$$f_2^t(P^{n_2-r_2,0}(f_2))\cap P^{2i_0m}(f_2)\ne \emptyset .$$ Then 
$$f_2^t(P^{n_2-r_2,0}(f_2))\cap P^{(3i_0+1)m}(f_2)=\emptyset ,$$ 
and 
$$m_2'+r_1/2<t<m_2-3i_0m,$$
 since $i_0m\le i_1m/10=r_1/10$.  Since $f_2^{t}(P^{n_2-r_2,0}(f_2))$ is a union of $\le i_0$ sets of ${\mathcal{P}}_{m_2-t}(f_2)$, we have 
$$f_2^{t}(P^{n_2-r_2,0}(f_2))\subset P^{0}(f_2)\setminus S_{m,f_2}^{3i_0+1}(P^0(f_2)),$$
 and hence
\begin{equation}\label{3.10.9}P^{n_2-r_2,0}(f_2)\subset P^{m_2'}(f_2)\setminus S_{m,f_2}^{3i_0+1}(P^{n_1}(f_2)).\end{equation}
In particular this gives (\ref{3.10.6}). Now we consider (\ref{3.10.11}), which it suffices to prove for $h=g$. Since $g^{m_2-k}(P^{n_2-r_2,0}(g))$ does not contain $g^i(S_{m,g}(P^0(g))$ for any $0<i\le m$, the same is true for $g^j(P^{n_2-r_2,0}(g))$ for all $0\le j\le m_2-k$. Similarly, since $g^{m_2-k}(P^{n_2-r_2,0}(g))$ does not contain $g^i(S_{m,g}^n(P^0(g))$ for any $0<i\le m$ and $n\ge 1$, the same is true for $g^j(P^{n_2-r_2,0}(g))$ for any $0\le j\le m_2-k$.  So now suppose that 
$$g^j(P^{n_2-r_2,0}(g))\cap P^{n_2-r_2,0}(g)\ne \emptyset$$
 for some $0<j<m_2$. Then  $0<j\le m_2-k$. Since $v_2(g)\in P^{n_1}(g)$, we know that 
$$P^{n_2-r_2,0}(g)\cap P^{n_1}(g)\ne \emptyset $$ and hence 
$$P^{n_2-r_2,0}(g)\subset P^{n_1-i_0}(g)$$
and then since $g^j(P^{n_2-r_2,0}(g))$ does not contain $P^{n_1-i_0}(g)$ we have
$$g^j(P^{n_2-r_2,0}(g))\subset P^{n_1-2i_0}(g)\subset P^{n_1-r_1}(g)$$
and hence $g^j(v_2(g))\in P^{n_1-r_1}(g)$, contradicting the definition of $m_2$. So (\ref{3.10.11}) is proved.  Then 
$$g^{jm_2}(v_2(g))\in S_{m_2,g}^{j_2-j}(P^{n_2-r_2,0}(g))\setminus S_{m_2,g}^{j_2+1-j}(P^{n_2-r_2,0}(g)),\ 0<1\le j_2$$
and so by a similar argument we see that $g^i(v_2(g))\notin P^{n_2-r_2,1}(g)$ for $jm_2<i\le (j+1)m_1$ for all $0\le j<j_2-1$ and hence, by the definition of $m_3$, we have (\ref{3.10.1}). 

 Now any two sets $g^i(S_{m_3,g}(P^{n_2-r_2,1}(g))$ and $g^j(S_{m_3,g}(P^{n_2-r_2,1}(g)))$ for $0\le i\le j\le m_3$ are either disjoint or the first is contained in the other, by the Markov property, which holds for $g$ as it holds for $f_2$. So (\ref{3.10.2}) holds for $h=g$, and hence for all $h\in V(P^0,\cdots P^{n_2},g)$. Similarly, (\ref{3.10.3}) holds except possibly for $i=m_2$, which, again, we only need to consider for $h=g$. So if (\ref{3.10.3}) does not hold for $h=g$ and $i=m_2$ then we have 
$$S_{m_3,g}(P^{n_2-r_2,1}(h))\subset g^{m_2}(S_{m_3,g}(P^{n_2-r_2,1}(g))))\subset P^{n_2-r_2,1}(g).$$ 
But then applying $g^{m_3-m_2}$ to the first inclusion, we see that $g^{m_3-m_2}(v_2)\in P^{n_2-r_2,1}(g)$,  and hence $m_3=2m_2$ and $S_{m_3,g}=S_{m_2,g}^2$ and $j_2\ge 1$, which is a contradiction. So (\ref{3.10.2}) holds also.

Now we construct $Q_3'$ and $P^{n_3-r_3,0}(f_2)$.
 Enlarging  $\varphi _2^{-1}(f_{2}^{m_2-k}(P^{n_2-r_2,1}(f_2)))$ to a union $Q_3''$ of two intervals. again, assuming again   that inverse images under $\varphi $ (and hence also under $\varphi _2$) of sets  in the partition ${\mathcal{P}}_0(f_1)$ are sufficiently close to subsets of $\{ z:|z|\le 1\}$  and  $\{ z:|z|\le 1\}$, we obtain that $Q_3''$ is disjoint from $\varphi _2^{-1}(f_2^{i}(P^m(f_2)))$ for $0\le i<m$. Again making suitable assumptions about the sets in ${\mathcal{P}}_0(h)$ and hence also of sets in ${\mathcal{P}}_k(h)$ we can, as before, assume that the union of sets of ${\mathcal{P}}_k(f_2)$ which intersect $\varphi _2(Q_3'')$ extending to  include complementary componenents  in a small neighbourhood and hence form a closed toological disc, is disjoint from $f_2^{i}(P^m(f_2))$ for $0\le i<m$.  We call this union of sets $f_2^{m_2-k}(P^{n_3-r_3,0})$. Pulling back  $Q_3''$ under $s^{m_2-k-m_2'}$ to $Q_3'''$,  and $f_2^{m_2-k}(P^{n_3-r_3,0})$ under the appropriate local inverse of $f^{m_2-k-j_12m+i_1m}$ we obtain 
 $$Q_3'''\subset \varphi _2^{-1}(f_2^{m_2'm}(P^{n_3-r_3,0}(f_2)))\subset \varphi_2^{-1}(P^0(f_2)),$$
  and then pulling back $f_2^{(j_1-i_1)m}(P^{n_3-r_3,0}(f_2))$ under $S_{m,f_2}^{j_1-i_1}$ to $P^{n_3-r_3,0}(f_2)$ gives (\ref{3.10.7}, since by its definition $P^{n_3-r_3,0}(f_2)$ is a larger union of sets than $P^{n_2-r_2,0}(f_2)$. Applying $S_{m_2,f_2}^{j_2-1}$ then gives (\ref{3.10.6}) again.  Then pulling back $Q_3'''$ under $s^{m_2'}$ to $Q_3'$ gives
\begin{equation}\label{3.10.8}Q_3'\subset \varphi _2^{-1}(P^{n_3-r_3,0}(f_2)),\end{equation} 
and this gives (\ref{3.10.5}) in this case. We also have 
$$P^{n_3-r_3,0}(f_2)\subset P^{m_2'}(f_2)\subset P^{n_1-r_1}(f_2),$$
and in exactly the same way as (\ref{3.10.11}), we obtain (\ref{3.10.12}).
\end{proof}

\section{The general step in the induction}\label{6}

\subsection{}\label{6.1}

 Now we need to show that the inductive step works after the second step. So we need to choose rational maps $f_\ell $, sets  $Q_\ell \subset S^1$ and points $e^{2\pi iw_\ell },\ e^{2\pi iy_\ell }\in Q_\ell $,  a continuous surjective map $\varphi _\ell $, approximated arbitrarily closely by homeomorphisms, where we will show that 
$$\varphi _\ell \circ (s_{y_\ell }\Amalg s_q)=f_\ell \circ \varphi _\ell .$$
As before, we define $r_1=i_1m$.
 For $\ell \ge 2$, let $m_\ell '$ be the greatest integer $i<m_\ell $ such that $g^i(v_2)\in P^{r_1}(g)$. We define
$$n_2-r_2=(j_2-1)m_2+m_2',$$
as before, and for $\ell \ge 3$,
  $$n_\ell -r_\ell =j_\ell m_\ell +m_{\ell -1}'.$$
As in the case $\ell =1$, we want every component of $Q_\ell $ to intersect $\varphi _\ell ^{-1}(F_2(f_\ell ))$. We also want
\begin{equation}\label{6.1.1}\varphi _\ell ^{-1}(P^{n_\ell }(f_\ell ))\cap S^1\subset Q_\ell \subset \varphi _\ell ^{-1}(P^{n_\ell -r_\ell }(f_\ell ))\cap S^1,\end{equation}
apart from some exceptional cases, where we require
\begin{equation}\label{6.1.3}P^{n_\ell ,1}(h)\subset P^{n_\ell -r_\ell ,1}(h)\end{equation}
for all $h\in V(P^0,\cdots P^{n_\ell},g)$ and 
\begin{equation}\label{6.1.2}\varphi _\ell ^{-1}(P^{n_\ell ,1}(f_\ell ))\cap S^1\subset Q_\ell \subset \varphi _\ell ^{-1}(P^{n_\ell -r_\ell ,1}(f_\ell ))\cap S^1,\end{equation}
The exceptional cases are $\ell =2$ and $\ell =3$ and $\ell \ge 4$, where $\ell -1$ is an {\em{exceptional value}}. We say that $2$ is an {\em{exceptional value}}, and $3$ is an {\em{exceptional value}} if 
$$S_{m_3,g}=S_{m_2,g}^{j_2}S_{m,g}^t$$
for some $t\le i_01$. Inductively, we say that $\ell \ge 4$ an {\em{exceptional value}} if $i$ is an exceptional value for $2\le i<\ell $ and
$$S_{m_\ell,g}=S_{m_{\ell -1},g}^{j_{\ell -1}}S_{m_{\ell -2},g}.$$

As in \ref{2.6} we define 
$$P^{n_{\ell +1}}(h)=S_{m_{\ell +1},h}^{j_{\ell +1}}(P^{n_\ell -r_\ell }(h)),$$
and 
$$P^{n_{\ell +1},1}(h)=S_{m_{\ell +1},h}^{j_{\ell +1}}(P^{n_\ell-r_\ell,1}(h)),$$
if we need $P^{n_\ell -r_\ell ,1}(h)\ne P^{n_\ell -r_\ell }(h)$. 

 In such cases $P^{n_\ell -r_\ell ,1}(h)$ is a closed topological disc which is a union of sets of ${\mathcal{P}}_{n_\ell -r_\ell +m_2-m_2'}(h)$ with
$$P^{n_\ell -r_\ell ,1}(h)\subset P^{n_\ell -r_\ell }(h)$$
 and hence  similar properties hold for $P^{n_{\ell +1},1}(h)$. We will also require a Markov property for $P^{n_\ell -r_\ell ,1}(g)$: that any components of $g^{-i}(P^{n_\ell -r_\ell ,1}(g))$ and $g^{-j}(P^{n_{\ell ' }-r_{\ell'} ,1}(g))$ for any $i$, $j\ge 0$, and any $\ell $, are either disjoint, or one is contained in the other, if either $\ell =\ell '=2$, which is a case already considered, or $\ell \ge \ell '\ge 3 $ and $\ell -1$ is an exceptional value.  This will follow from similar Markov properties for the sets $P^{n_2-r_2,0}(h)$ and $P^{n_3-r_3,0}(h)$, since for $\ell \ge 3$, either  $P^{n_\ell -r_\ell ,1}=P^{n_\ell -r_\ell }(h)$  or , if $\ell -1$ is an exceptional value then $P^{n_{\ell } -r_{\ell},1}(h)$ will be a component of $h^{m_2'-n_\ell +r_\ell}(P^{n_3-r_3,0}(h))$.    We define 
$$P^{n_\ell -r_\ell ,1}(h)=P^{n_\ell -r_\ell }(h),\ \ \ P^{n_{\ell +1},1}(h)=P^{n_{\ell+1} }(h)$$
 if $\ell -1$ is not an exceptional value. 

As in \ref{2.6}, we define $j_\ell $, for all $\ell \ge 1$,  to be the largest integer $\ge 1$ such that $v_2(g)\in S_{m_\ell ,g}^{j_\ell }(P^{n_{\ell -1}-r_{\ell -1},1}(g))$,  and then 
$$n_\ell =n_{\ell -1}-r_{\ell -1}+j_\ell m_\ell $$
 for $\ell \ge 1$, remembering that $n_0=r_0=0$. So
 $$v_2(g)\in P^{n_\ell ,1}(g)\setminus S_{m_\ell ,g}(P^{n_\ell,1}(g)).$$
Then we define $m_{\ell +1}$, for $\ell \ge 2$, to be the least integer 
 $>j_\ell m_\ell $ and with 
 $$g^{m_{\ell +1}}(v_2(g))\in P^{n_\ell -r_\ell ,1}(g).$$ 
We need the following.
 
 \begin{lemma}\label{3.12} 
 
For $\ell \ge 2$,
 \begin{equation}\label{3.12.2}g^i(v_2(g))\in P^{n_\ell -r_\ell ,1}(g)\Rightarrow i=0\mbox{ or }i=m_\ell\mbox{ or }i\ge m_{\ell +1}.\end{equation}
 For $h\in V(P^0,\cdots P^{n_{\ell }-r_\ell +m_{\ell +1}},g)$, 
\begin{equation}\label{3.12.3}h^i(S_{m_{\ell +1},h}(P^{n_\ell -r_\ell,1}(h)))\cap S_{m_{\ell +1},h}(P^{n_\ell -r_\ell,1}(h))=\emptyset,\ 0<i<m_{\ell +1},\end{equation}
\begin{equation}\label{3.12.4}h^i(S_{m_{\ell +1},h}(P^{n_\ell -r_\ell,1}(h)))\cap P^{n_\ell -r_\ell,1}(h)=\emptyset,\ 0<i<m_{\ell +1},\ i\ne m_\ell,\end{equation}
\begin{equation}\label{3.12.0}\begin{array}{ll}n_3-r_3>n_2+(j_3-1)m_3&\ \\
P^{n_{3}-r_{3}}(h)\subset S_{m_{3},h}^{j_{3}-1}(P^{n_2 }(h)), &\ \\
n_{\ell +1}-r_{\ell +1}\ge n_\ell +(j_{\ell +1}-1)m_{\ell +1}+m_{\ell -1}&\mbox{ if }\ell \ge 3 ,\\ 
P^{n_{\ell +1}-r_{\ell+1}}(h)\subset S_{m_{\ell +1},h}^{j_{\ell +1}-1}(P^{n_\ell ,1}(h))&\mbox{ if }\ell \ge 3.\end{array}\end{equation}
 \begin{equation}\label{3.12.1}P^{n_{\ell+1} -r_{\ell+1},1}(h)\cap h^i(P^{n_{\ell+1} -r_{\ell+1},1}(h))=\emptyset , 0<i<m_{\ell +1}\end{equation}
and for $\ell \ge 0$
\begin{equation}\label{3.12.5}m_{\ell+3}\ge j_{\ell +2}m_{\ell +2}+m_{\ell +1}\end{equation}
\end{lemma}
\begin{proof} 

These are proved inductively. By (\ref{3.10.5}) to (\ref{3.10.2}) we have (\ref{3.12.2}) to (\ref{3.12.4}) for $\ell =2$. 

$g^i(v_2(g))\notin P^{n_\ell -r_\ell}(g)$ for $i<m_\ell $, $i\ne m_{\ell -1}$ and  by the inductive hypothesis of (\ref{3.12.0}) we also have  $g^{m_{\ell -1}}(v_2)\notin P^{n_\ell -r_\ell}(g)$. Since $P^{n_\ell -r_\ell,1}(g)\subset S_{m_\ell,g}^{j_\ell -1}(P^{n_{\ell -1}-r_{\ell -1},1}(g))$ we also have $g^{im_\ell}(v_2)\notin P^{n_\ell -r_\ell,1}(g)$ for $1<i\le j_\ell $. If $j_\ell >1$ then by the inductive hypothesis, since $P^{n_\ell -r_\ell ,1}=S_{m_\ell ,g}^{j_\ell -j}(P^{n_{\ell -1}-r_{\ell -1},1}(g))$ for $1\le j<j_\ell $, the sets $g^i(P^{n_\ell -jm_\ell,1 }) $, for $0\le i<m_\ell $ are disjoint, and hence $g^{jm_\ell +i}(v_2)\notin P^{n_\ell -r_\ell ,1}(g)$ for $0< i\le m_\ell $ and $1\le j<j_\ell $. This and the definition of $m_{\ell +1}$ give (\ref{3.12.2}).

If (\ref{3.12.3}) does not hold for $i$ and $g$ then we have, by the Markov property for $P^{n_\ell -r_\ell ,1}(g)$, 
$$S_{m_{\ell +1},g}P^{n_\ell -r_\ell ,1}(g)\subset g^i(S_{m_{\ell +1},g}(P^{n_\ell -r_\ell ,1}(g))\subset P^{n_\ell -r_\ell ,1}(g)$$
and hence $g^i(v_2(g))\in P^{n_\ell -r_\ell ,1}(g)$ and $g^{m_{\ell +1}-i}(v_2(g))\in P^{n_\ell -r_\ell ,1}(g)$. So $i=m_{\ell +1}-i=m_\ell $ by (\ref{3.12.2}) and $S_{m_{\ell +1}}=S_{m_\ell }^2$, which we know is not the case. So (\ref{3.12.3}) holds for $g$ and hence also for $h\in V(P^0,\cdots P^{n_\ell -r_\ell +m_{\ell +1}},g)$. Then (\ref{3.12.4}) follows from (\ref{3.12.2}).

Now we consider (\ref{3.12.0}). We have
$$n_3-r_3=j_3m_3+m_2'>(j_3-1)m_3+j_2m_2+n_1-r_1=n_2+(j_3-1)m_3,$$
which gives (\ref{3.12.0}) in the case $\ell =2$. If $\ell =3$ we have, using (\ref{3.12.5}) with $1$ replacing $\ell $ (to be done later), and also $m_3'\ge (j_2-1)m_2+m_2'$,
$$n_4-r_4=j_4m_4+m_3'\ge (j_4-1)m_4+j_3m_3+m_2+(j_2-1)m_2+m_2'$$
$$=(j_4-1)m_4+j_3m_3+j_2m_2+m_2'$$
$$ n_3=j_3m_3+n_2-r_2=j_3m_3+(j_2-1)m_2+m_2',$$
and so 
  $$n_4-r_4\ge (j_4-1)m_4+n_3+m_2.$$
Now since $P^{n_3,1}(h)$ is a union of sets of ${\mathcal{P}}_{n_3+m_2-m_2'}(h)$ 
we have (\ref{3.12.0}) for $\ell =3$ as required. 

Now we want to prove (\ref{3.12.0}) for $\ell \ge 4$, assuming that (\ref{3.12.0}) is true for $2\le \ell '<\ell $. By (\ref{3.12.5}) with $\ell -2$ replacing $\ell $ (to be done later) we have
$$m_{\ell +1}\ge j_\ell m_\ell +m_{\ell -1}.$$ 
Then
$$n_{\ell+1} -r_{\ell +1}=j_{\ell +1}m_{\ell +1}+m_{\ell }'\ge j_{\ell +1}(j_\ell m_\ell +m_{\ell -1})+m_\ell '.$$
But $m_\ell '\ge j_{\ell -1}m_{\ell -1}+m_{\ell-2}'$. So
$$n_{\ell+1} -r_{\ell +1}\ge (j_{\ell +1}-1)m_{\ell +1}+j_\ell m_\ell +(j_{\ell -1}+1)m_{\ell -1}+m_{\ell -2}',$$
while
$$n_\ell =j_\ell m_\ell +j_{\ell -1}m_{\ell -1}+m_{\ell -2}'.$$
 It follows, as in the case $\ell =3$, that
$$P^{n_{\ell +1}-r_{\ell +1}}(h)\subset S_{m_{\ell +1},h}^{j_{\ell+1}-1}(P^{n_\ell ,1} (h)),$$
as required for (\ref{3.12.0}).

Now we consider (\ref{3.12.1}), which we need to prove for all $\ell \ge 2$. By (\ref{3.12.0}) we have, for all $\ell \ge 2$,
$$P^{n_{\ell +1}-r_{\ell +1},1}(h)\subset S_{m_{\ell +1},h}^{j_{\ell +1}-1}(P^{n_\ell -r_\ell ,1}(h))$$

 Now (\ref{3.12.3}) (or (\ref{3.10.2}) if $\ell =2$)  gives (\ref{3.12.1}) with $P^{n_{\ell +1}-r_{\ell +1},1}(h)$ replaced by  the set $S_{m_{\ell +1},h}(P^{n_\ell -r_\ell ,1}(h))$. 
  If $j_{\ell +1}>1$ then this is enough to give (\ref{3.12.1}). So now suppose that $j_{\ell +1}=1$. 

 If (\ref{3.12.1}) does not hold then for $i$ then by the Markov property for $P^{n_{\ell +1}-r_{\ell +1},1}(h)$,
 $$P^{n_{\ell +1}-r_{\ell +1},1}(h)\subset h^i(P^{n_{\ell +1}-r_{\ell +1},1}(h))$$. 
 So
$$h^i(P^{n_{\ell +1}-r_{\ell +1},1}(h))\cap P^{n_{\ell}-r_\ell,1}(h)\ne \emptyset .$$
Then  by the Markov property, if $\ell -1$ is nonexceptional  or $\ell \ge 3$ is exceptional, 
$$ h^i(P^{n_{\ell +1}-r_{\ell +1},1}(h))\subset  P^{n_{\ell}-r_\ell,1}(h)$$
and then 
$$h^i(v_2)\in P^{n_\ell -r_\ell}(v_2)\Rightarrow i=m_\ell.$$
by (\ref{3.12.2}). If $\ell -1$ is an exceptional value and $\ell $ is non-exceptional,  we can only deduce $i=m_\ell $ for $i\le m_{\ell +1}-r_{\ell +1}-(m_2-m_2')$ because then
$$h^i(P^{n_{\ell +1}-r_{\ell +1},1}(h))\cap P^{n_{\ell +1}-r_{\ell +1},1}(h)\ne \emptyset \Rightarrow h^i(P^{n_{\ell +1}-r_{\ell +1}}(h))\cap P^{n_{\ell }-r_{\ell },1}(h)\ne \emptyset$$
$$\Rightarrow h^i(P^{n_{\ell +1}-r_{\ell +1}}(h))\subset P^{n_{\ell +1}-r_{\ell +1}-i}(h)\subset P^{n_\ell -r_\ell ,1}(h)$$
because $P^{n_\ell -r_\ell ,1}(h)$ is a union of sets of ${\mathcal{P}}_{n_\ell -r_\ell +m_2-m_2'}(h)$ and 
$$n_\ell -r_\ell +m_2-m_2'=n_{\ell +1}-m_{\ell +1}+m_2-m_2'\le n_{\ell +1}-r_{\ell +1}-i.$$

Now if (\ref{3.12.1}) does not hold for $i$ we have
\begin{equation} P^{n_{\ell +1}-r_{\ell +1}}(h)\subset h^i(P^{n_{\ell +1}-r_{\ell +1}}(h))\label{6.2.20}\end{equation}
  We write, still assuming $j_{\ell +1}=1$,
$$P^{n_{\ell +1}-r_{\ell +1}}(h)=S_{m_{\ell +1},h}S'P^0(h),$$
where 
$$S'=S_{m_\ell ,h}h^{m_\ell -m_\ell '},$$
We write
$$S_{m_{\ell +1},h}=S_{m_\ell ,h}^{j_\ell }U=S^{j_\ell }U$$
and thus $S'$ is a prefix of $S$. Then
$$P^{n_{\ell +1}-r_{\ell +1}}(h)=S^{j_\ell }US'P^0(h).$$
Then (\ref{6.2.20}) is equivalent to showing that certain suffixes of $S^{j_\ell }US'$ are not also prefixes.   We also know by the definition of $j_\ell $ and $m_{\ell +1}$ that $S$ is not a prefix of $U$ and also not a suffix of $U$. We only need to consider suffixes of the form $S^{j_\ell -1}US'$ (for $i=m_\ell $) and suffixes of the form $U_1S'$ for $U_1$ a  suffix of $S^{j_\ell }U$ with $|U_1|\le r_{\ell +1}$, except when $\ell =2$ or $\ell =3$ and $3$ is not exceptional or $\ell \ge 4$ and $\ell -1$ is exceptional, when it suffices to consider $i\le r_{\ell +1}+m_2-m_2'$.   Now $i=m_\ell $ is impossible if $|U|\ge |S|$ because then $S$ must be a prefix of $U$, giving a contradiction. So for that case we only need to consider $|U|<|S|$.  The question is then whether $US'$ is a prefix of $SUS'$, that is, whether $US'$ is a prefix of $SU$.

Now we consider $|U_1|\le r_{\ell +1}$. We have 
$$r_3=(j_2-1)m_2<j_2m_2$$
and for $\ell \ge 3$
$$r_{\ell +1}=j_{\ell }m_{\ell }+m_{\ell -1}'-m_\ell '\le j_\ell m_\ell .$$
So we have $|U_1|\le j_\ell m_\ell $ We must then have $U_1=S^rU_2$ for some $r\le j_\ell -1$ and $0<|U_2|<|S|$ because $U$ does not have $S$ as a suffix. Then $U_2S'$ is a prefix of $S^{j_\ell -r}U$. If $j_\ell -r=1$ we obtain that $U_2S'$ is a prefix of $SU$, and if $j_\ell -r\ge 2$ we obtain that $U_2S'$ is a prefix of $S^2$. In the latter case, that is, $r\le j_\ell -2$ we obtain that $U_2S'$ is a prefix of $SU_2$. If $r=j_\ell -1$, 
$$|U_2|=|U_1|-(j_\ell-1)m_\ell \le m_\ell -m_\ell '+m_{\ell -1}'$$
and $U_2S'=SU'$ with 
$$|U'|=|U_2|+|S'|-|S|\le m_{\ell -1}'$$
Now since $P^{n_{\ell +1}-r_{\ell +1}}(h)\subset P^{n_\ell }(h)=S^{j_\ell }P^{n_{\ell -1}-r_{\ell -1}}(h)$ we see that $S_{m_{\ell-1},h}'$ is a prefix of $U$ for all $\ell \ge 3$ and $S_{m_{\ell -1},h}^{j_{\ell -1}}S_{m_{\ell -2},h}'$ is a prefix of $U$ for all $\ell \ge 4$. So  $U'$ is a prefix of $S$ and hence also of $U_2$, since $|U'|<|U_2|$. 

Now we need to consider $|U_1|\le r_{\ell +1}+m_2-m_2'$ if $\ell =2$ or $\ell =3$ is exceptional or $\ell \ge 4$ and $\ell -1$ is exceptional. We have
$$r_3+m_2-m_2'=(j_2-1)m_2+m_2-m_2'=j_2m_2-m_2',$$
If $3$ is not an exceptional value then $m_3'\ge m_2$ so
$$r_4+m_2-m_2'\le j_3m_3+m_2-m_3'\le j_3m_3.$$
If $\ell \ge 4$ and $\ell -1$ is exceptional then we can prove by induction on $i$ that $m_i-m_i'\le m_2-m_2'$ for all $2\le i\le \ell -1$. This is obviously true for $i=2$, and is true for $i=3$ because $m_3'=m_2'$. For $4\le i\le \ell -1$ we have $m_i=j_{i-1}m_{i-1}+m_{i-2}$ and $m_i'=j_{i-1}m_{i-1}+m_{i-2}'$, so $m_i'-m_i=m_{i-2}'-m_{i-2}$ and the induction is completed. This gives 
$$r_{\ell +1}+m_2-m_2'=j_{\ell }m_{\ell }+m_{\ell -1}'-m_{\ell -1}+m_{\ell -1}-m_\ell '+m_2-m_2'$$
$$=j_{\ell }m_{\ell }+m_{\ell -1}-m_\ell '\le j_\ell m_\ell$$
Arguing as in the case of $|U_1|\le r_{\ell +1}$ we obtain that $U_1=S^rU_2$ with $U_2S$ a prefix of $SU_2$ or $r=j_\ell -1$ and $|U_2|\le m_{\ell -1}$. This gives $U_2S'=SU_2'$ for a prefix $U_2'$ of $U_2$ if $\ell \ge 4$. We leave aside for the moment the case $\ell =2$ and $\ell =3$ if $3$ is an exceptional value. 
 
 So now suppose that    $0<|U_1|<|S|$ and that  $U_1S'$ is not a prefix of $SU$, that is $U_1S'$ is a prefix of $SU_1$. From this point on, $U$ itself is irrelevant and we are rewriting $U_2=U_1$ in the later cases above. Now if $U_1S'$ is  a prefix of $S$ with $|U_1|\le n_1$ then we deduce from this that $S'=S_{m,h}^qT$ for  $T$ a (possibly trivial) prefix of $S_{m,h}$, which is only possible at all if $\ell =3$ is exceptional, and we are leaving this case aside for the moment. If $|U_1|>n_1$ then $S'=U_1^qU_1'$ for a prefix $U_1'$ of $U_1$ and some $q\ge 0$  and $U_1=U_1'U_1''$, and $U_1S'=U_1^{q+1}U_1'=S'U_1''U_1'$ is a prefix of $S$, contradicting the definition of $S'$, since $U_1''U_1'$ must contain $S_{m,h}^{i_1}$ as a subword,  because $U_1$ starts with $S_{m,h}^{j_1}$.  So now we assume that 
 $$U_1S'=SU_1'$$
  for $|U_1|<|S|$ and $S'$, $U_1$, $U_1'$ all prefixes of $S$, and we will obtain a contradiction, apart from some exceptions.  Note that if $S'=S$ then $S=U_1V_1$ and $U_1U_1V_1=U_1V_1U_1$ gives $U_1V_1=V_1U_1$, a contradiction because then $\bigcap _{n\ge 0}S^nP^{n_\ell }$ is a point of  period $<m_\ell $. 

We change notation and write $A_1=S$, $A_1'=S'$, $A_2=U_1$ and $A_2'=U_1'$ where all of $A_1'$, $A_2$ and $A_2'$ are prefixes of $A_1$ and $A_2'$ is also a prefix of $A_2$ and we have
\begin{equation}\label{3.12.7}A_1A_2'=A_2A_1'\end{equation}
We have
$$|A_1|-|A_1'|=|A_2|-|A_2'|,\ \ \ |A_2|<|A_1|$$
Now write
$$A_1=A_2A_{3,1}$$
Then 
$$A_2A_{3,1}A_2'=A_2A_1'$$
and 
$$A_{3,1}A_2'=A_1'$$
Either $A_1'=A_2A_{3,1}'$ for a prefix $A_{3,1}'$ of $A_{3,1}$ or $|A_1'|<|A_2|$. In the former case we  have
$$A_{3,1}A_2'=A_2A_{3,1}',$$
and then either $|A_2|\le |A_{3,1}|$ in which case we can write $A_{3,1}=A_2A_{3,2}$ and can continue. In the latter case we write $r_1=1$ and $A_{3,1}=A_3$ and
$$A_1=A_2^{r_1}A_3,\ \ A_1'=A_2^{r_1-1}A_3.$$
In general we obtain an integer $r_1>0$ and $A_3$, $A_3'$ such that either
$$A_1=A_2^{r_1}A_3,\ \ A_1'=A_2^{r_1}A_3',\ |A_3'|<|A_3|<|A_2|,$$
or
$$A_1=A_2^{r_1}A_3,\ \ A_1'=A_2^{r_1-1}A_3', |A_3|< |A_3'|<|A_2|$$
$$|A_1|-|A_1'|=|A_2|+|A_3|-|A_3'|$$
In the first case (\ref{3.12.7}) gives
$$A_2^{r_1}A_3A_2'=A_2^{r_1+1}A_3'$$
and hence
$$A_3A_2'=A_2A_3'$$
and we can continue to find $r_2$ and $A_4$. In the second case we obtain
$$A_2^{r_1}A_3A_2'=A_2^{r_1}A_3'$$
and hence 
$$A_3A_2'=A_3'$$
and
$$|A_3|+|A_2'|=|A_3'|.$$
If $A_3$ is trivial in this case we have $A_1=A_2^{r_1}$ and since $|A_2|<|A_1|$ we have $r_1>1$ which gives a contradiction. Since $A_i$ is decreasing with the first case we must reach the second case for some $N\ge 1$ that is
$$A_{i+1}A_i'=A_iA_{i+1}',\mbox{ for }i\le N$$
$$A_i=A_{i+1}^{r_i}A_{i+2},\ A_i'=A_{i+1}^{r_i}A_{i+2}'\mbox{ for }i<N,$$
$$|A_i|-|A_i'|=|A_1|-|A_1'|\mbox{ for }i\le N+1$$
\begin{equation}\label{3.12.8}A_{N+2}A_{N+1}'=A_{N+2}',\end{equation}
\begin{equation}\label{6.2.9}A_{N}=A_{N+1}^{r_{N}}A_{N+2},\ \ A_N'=A_{N+1}^{r_{N}-1}A_{N+2}',\end{equation}
where
$$(r_{N}-1)|A_{N+1}|+|A_{N+2}'|=|A_N'|<r_{N}|A_{N+1}|< |A_N|$$
and
$$|A_N|-|A_N'|=|A_{N+2}|+|A_{N+1}|-|A_{N+2}'|>|A_{N+2}|.$$
All $A_i$, $A_i'$ are prefixes of $A_1$. $A_i$ is a suffix of $A_1$ for odd $i$ and a suffix of $A_2$ for even $i$. $A_i'$ is a suffix of $A_1'$ for all $i$. From (\ref{3.12.8}) we obtain
\begin{equation}\label{6.2.10}A_{N+2}'=A_{N+2}^{r_{N+1}}A_{N+3}',\ \ A_{N+1}'=A_{N+2}^{r_{N+1}-1}A_{N+3}'\end{equation}
where $A_{N+3}'$  is a prefix of $A_{N+2}$. 
Write
$$A_1=A_1'B_1, A_2=A_2'B_2.$$
$$A_i=A_i'B_i,\ \ i\le N+1.$$
Then $|B_i|=|B_1|$ for $i\le N+1$ and 
$$B_i=B_1\mbox{ if }i\mbox{ is odd,}$$ 
$$B_i=B_2\mbox{ if }i\mbox{ is even.}$$
Now we recall that $A_1=S_{m_\ell }$ and $A_1'=S_{m_\ell }'$. 
$B_1$ does not start with $S_m$ and does not have $S_m^{i_1}$ as a subword. $A_1'$ ends with $S_m^{i_1}$ and hence $S_m^{i_1}$ is a suffix of $A_i'$ for all  $i$. 
From $A_N=A_N'B_N$ and(\ref{6.2.9}) we obtain
$$A_{N+1}A_{N+2}=A_{N+2}'B_N.$$
Then
$$A_{N+1}'B_{N+1}A_{N+2}=A_{N+2}'B_N.$$
Then from (\ref{6.2.10}), 
$$A_{N+2}^{r_{N+1}-1}A_{N+3}'B_{N+1}A_{N+2}=A_{N+2}^{r_{N+1}}A_{N+3}'B_N$$
and
$$A_{N+3}'B_{N+1}A_{N+2}=A_{N+2}A_{N+3}'B_N.$$
So writing $A_{N+2}=A_{N+3}'A_{N+3}''$ we have
$$B_{N+1}A_{N+2}=A_{N+3}''A_{N+3}'B_N.$$
$|B_i|>|A_{N+2}|$ for all $i\le N+1$, so there is a subword $C$ such that
$$B_{N+1}=A_{N+3}''A_{N+3}'C,\ \ B_N=CA_{N+2}.$$
So one of $A_{N+2}$ and $A_{N+3}''A_{N+3}'$ does not contain a subword $S_m^{i_1}$. But --- recalling again that $A_{N+2}$ is a prefix of $A_1=S_{m_\ell ,h}$ --- $A_{N+2}$ starts with $S_m^{j_1}$ or is of shorter length that this. If it starts with $S_m^{2j}$ then one of $A_{N+3}''$ or $A_{N+3}'$ contains a subword $S_m^j$. So $A_{N+2}$ is a proper prefix of $S_m^{2i_1}$. Then it must be a proper prefix of $S_m^{i_1}$ because otherwise both $A_{N+2}$ and $A_{N+3}''A_{N+3}'$ contain $S_m^{i_1}$ and a subword and hence  both $B_N$ and $B_{N+1}$ contain $S_m^{i_1}$ as a subword. Now we have
$$A_{N+1}'=S_m^{p_{N+1}},$$
$$A_{N+1}=S_m^{p_{N+1}+q_{N+1}}C$$
This is only possible if $p_{N+1}+q_{N+1}=j_1$. 
So 
$$A_{N+1}=S_m^{j_1}C.$$
Then
$$A_{N}=(S_m^{j_1}C)^{r_N}S_m^{t_0}$$
where $A_{N+2}=S_m^{t_0}$ for some $t_0<i_1$. Then we get $A_i$ for all $i$. and $B_i=CS_m^{t_0}$ if $i$ is odd and   $S_m^{t_0}C$ if $i$ is even. In particular $N$ is odd or $t_0=0$. But if $t_0=0$ then $A_{N+2}$ is trivial and we are done. So assume $t_0>0$ and $N$ is odd
$$A_N'=(S_m^{j_1}C)^{r_N-1}S_m^{j_1}.$$
Then
$$A_{N-1}=A_N^{r_{N-1}}A_{N+1}$$
$$A_{N-1}'=A_N^{r_{N_1}}A_{N+1}'$$

Now we recall again that $A_1=S=S_{m_\ell,h}$ for a suitable $h$. With this choice we have $S_{m_2,h}=A_{N+1}$ and $S_{m_3,h}=A_N$ and thus $N=\ell -2$ and we have
$$S_{m_i}=A_{N+2+i-\ell},\ \ j_i=r_{N+2+i-\ell }$$
$$S_{m_3}=S_{m_2}^{j_2}S_m^{t_0},$$
$$S_{m_{i+1},h}=S_{m_i,h}^{j_{i}}S_{m_{i-1},h},\ \ 3\le i\le \ell -1.$$
So except in this case,  and the case $\ell =2$, the proof of (\ref{3.12.1}) is completed. Note that the assumption that $N$ is odd can no longer be made because in these exceptional cases the inductive hypothesis, that (\ref{3.12.1}) is true for $\ell '<\ell $ replacing $\ell $, fails. So the proof is completed precisely except in the cases when $\ell $ is an exceptional value

 Now we consider (\ref{3.12.5}) --- leaving out the case of $\ell $ being an exceptional value, but including $\ell =0$ and $\ell =1$. From the fact that $h^i(P^0(h))$ are disjoint for $0\le i<m$ we have $m_3\ge j_2m_2+m$, which completes the case $\ell =0$. So now we consider $\ell \ge 1$. We have $g^{m_3}(v_2(g))\in P^{n_2-r_2,1}(g)$ and $g^{m_{\ell +2}}(v_2(g))\in P^{n_{\ell +1}-r_{\ell +1}}(g)$ for $\ell \ge 2$. Then by the definition of $j_{\ell +2}$ we have $v_2(g)\in S_{m_3,g}^{j_3}(P^{n_2-r_2,1}(g)$ and hence $g^{j_3m_3}(v_2(g))\in P^{n_2-r_2,1}(g)$ and similarly $g^{j_{\ell +1}m_{\ell +1}}(v_2(g))\in P^{n_{\ell +1}-r_{\ell +1}}(g)$ for $\ell \ge 2$. But by (\ref{3.10.11}) the sets $g^i(P^{n_2-r_2,1}(g))$ are disjoint for $0\le i<m_2$, and by (\ref{3.12.1}), for $\ell \ge 2$, the sets $g^i(P^{n_{\ell +1}-r_{\ell +1}}(g))$ are disjoint from $P^{n_{\ell +1}-r_{\ell +1}}(g)$ for $0<i<m_{\ell +1}$, replacing $P^{n_3-r_3}(g)$ by $P^{n_3-r_3,1}(g)$ in the case $\ell =2$. In particular, $g^{i+j_3m_3}(v_2(g))\notin P^{n_2-r_2,1}(g)$ for $0<i<m_2$ and   $g^{i+j_{\ell +2}m_{\ell +2}}(v_2(g))\notin P^{n_{\ell +1}-r_{\ell +1}}(g)$ for $0<i<m_{\ell +1}$. So for $\ell \ge 1$,
 $$m_{\ell +3}\ge j_{\ell +2}m_{\ell +2}+m_{\ell +1},$$
 giving (\ref{3.12.5}).

Now we consider the case of $\ell $ being an exceptional value.

The previous definitions can be restated as
$$P^{n_2-r_2,0}(h)=S_{m_2,h}(h^{m_2}(P^{n_2-r_2,0}(h))),$$
where $h^{m_2}(P^{n_2-r_2,0}(h))$ is a union of sets of ${\mathcal{P}}_0(h)$ which can be taken arbitrarily close to $P^0(h)$. Thus we are simply enlarging the domain of $S_{m_2,h}$. Note that by (\ref{3.10.11}), $S_{m_2,h}$ is  two-valued on this domain, as it is on $P^{n_2}(h)$. Similarly,
$$P^{n_3-r_3,0}(h)=S_{m_2,h}(h^{m_2}(P^{n_3-r_3,0}(h))),$$
and $S_{m_2,h}$ is again two-valued on this enlarged domain, by (\ref{3.10.12}).  Recall that
$$P^{n_2-r_2,1}(h)=S_{m_2,h}^{j_2-1}(P^{n_2-r_2,0}(h)).$$ 

Recall that from \ref{3.10} we have
$$\varphi _2^{-1}(P^{n_2}(f_2))\subset Q_2\subset \varphi _2^{-1}(P^{n_2-r_2,1}(f_2)),$$
$$\varphi _3^{-1}(P^{n_3,1}(f_3)\subset Q_3\subset \varphi _3^{-1}(P^{n_3-r_3,1}(f_3)).$$

 Also we have, for $4\le i\le \ell $, for some local inverse $U_{i,h}$,
$$S_{m_i,h}=U_{i,h}S_{m_2,h}\mbox{ if }i\mbox{ is even,}$$
$$S_{m_i,h}=U_{i,h}S_{m_2,h}S_{m,h}^{t_0}\mbox{ if }i\mbox{ is odd,}$$
as $S_{m_3,h}=S_{m_2,h}^{j_2}S_{m,h}^{t_0}$. In particular, $U_{2,h}$ is the identity and and $U_{3,h}=S_{m_2,h}^{j_2-1}$
Then, for all $3\le i<\ell$,
$$U_{i+1,h}=S_{m_{i},h}^{j_{i}}U_{i-1,h}.$$
Then for $3\le i\le \ell +1$, we  define
$$P^{n_i-r_i,1}(h)=S_{m_i,h}^{j_i}U_{i-1,h}(P^{n_3-r_3,0}(h)),$$
and  if $4\le i\le \ell +2$, we define
$$P^{n_i,1}(h)=S_{m_i,h}^{j_i}P^{n_{i-1}-r_{i-1},1}(h).$$

 Since $P^{n_2-r_2,1}(h)\subset P^{n_2-r_2}(h)$, we obtain $P^{n_i-r_i,1}(h)\subset P^{n_i-r_i}(h)$ for all $4\le i\le \ell $, and similarly we obtain $P^{n_i,1}(h)\subset P^{n_i}(h)$ and for $i\ge 2$,  $P^{n_i-r_i,1}(h)$ is a union of sets of ${\mathcal{P}}_{n_i-r_i+m_2-m_2'}(h)$ and $P^{n_i,1}(h)$ is a union of sets of ${\mathcal{P}}_{n_i+m_2-m_2'}(h)$. Also, by definition, $P^{n_i-r_i,1}(h)$ is a component of $h^{m_2'-n_i+r_i}(P^{n_3-r_3,0}(h))$. This property was used in the proof of (\ref{3.12.2}) to (\ref{3.12.0}), which therefore hold. 

For  (\ref{3.12.1}) for an exceptional $\ell $, we need that no suffix of any of the words $S_{m_{\ell+1},h}^{j_{\ell+1}}U_{\ell,h}S_{m_2,h}$, for varying $t$, is also a prefix. Arguing as in the nonexceptional case, if some suffix is a prefix then for $S=S_{m_\ell,h}$ and some proper prefix $U_1$ of $S$  we have
$$SU_1=U_1S.$$
Now arguing as in the nonexceptional case, we obtain a prefix $A$ of $A$ with $|A|\le |U_1|<|S|$ and an integer $n>1$ such that $S=A^n$. This gives a contradiction because it contradicts (\ref{3.12.3}) with $\ell -1$ replacing $\ell $ for example, that is, that $h^i(S_{m_\ell ,h}(P^{n_{\ell -1}-r_{\ell -1},1}(h)))$ are disjoint for $0\le i<m_\ell $. So now we have (\ref{3.12.1}) in all cases, and then we obtain (\ref{3.12.5}) as before. 

\end{proof}

 \subsection{Construction of $Q_{\ell +1}$} 
   
For $\ell \ge 2$, that is, $\ell +1\ge 3$, the set $Q_{\ell +1}$ is defined in terms of a set $Q_{\ell +1}'$. The construction for $\ell =2$ has already been done so now we consider $\ell \ge 3$, assuming that for $i\le \ell +1$ we have
$$\varphi _i\circ (s_{y_i}\Amalg s_q)=f_i\circ \varphi _i.$$
 We define $T_{m_{\ell +1}}$ to be the local inverse of $s^{m_{\ell +1}}$ which satisfies
$$\varphi _{\ell +1}\circ T_{m_{\ell +1}}=S_{m_{\ell +1},f_{\ell +1}}\circ \varphi _{\ell +1}.$$ 
Then we define 
$$Q_{\ell +1}=T_{m_{\ell +1}}^{j_{\ell +1}}(Q_{\ell +1}')$$
So we need to define $Q_{\ell +1}'$, for $\ell \ge 3$. We need $Q_{\ell +1}'$ to contain $\varphi _{\ell +1}^{-1}(P^{n_\ell -r_\ell ,1}(f_{\ell +1}))$.  The following is proved in exact analogue to \ref{3.40}. So the proof is omitted
\begin{lemma} \label{3.50}For any $n\ge 0$, let $R(g)$ be an isotopically varying set of ${\mathcal{P}}_n(g)$ which is not strictly contained in the backward orbit of $P^{n_\ell -r_\ell ,1}(g)$, for $g\in V(P^0,\cdots P^{n_\ell },f_\ell )$. Then
 $$\varphi _{\ell +1}^{-1}(R(f_{\ell +1}))=\varphi _{\ell}^{-1}(R(f_{\ell })).$$
   \end{lemma}
 In particular, we have

$$\varphi _{\ell +1}^{-1}(P^{n_\ell -r_\ell ,1}(f_{\ell +1}))=\varphi _\ell ^{-1}(P^{n_\ell -r_\ell,1}(f_\ell )),$$
Now $Q_{\ell +1}'$ is defined by adding to $\varphi _{\ell }^{-1}(P^{n_\ell -r_\ell,1}(f_{\ell }))\cap S^1$ shorter complementary components so that $Q_{\ell +1}'$ has just two components, both of which intersect the boundary of the minor gap of $L_{y_\ell }$.
So in order to proceed, we prove the following. 

\begin{lemma}\label{3.31} 
For $\ell \ge 3$,
\begin{equation}\label{3.31.2}\varphi _{\ell +1}^{-1}(P^{n_{\ell +1},1}(f_{\ell +1}))\subset Q_{\ell +1} \subset \varphi _{\ell +1}^{-1}(P^{n_{\ell +1}-r_{\ell +1},1}(f_{\ell +1})).\end{equation}
\end{lemma}
\begin{proof}
We have already defined $Q_{\ell +1}'$ and $Q_{\ell +1}=T_{m_{\ell +1}}^{j_{\ell +1}}(Q_{\ell +1}')$ so that the first inclusion of  (\ref{3.31.2}) holds. In order to bound $Q_{\ell +1}'$ -- and hence also bound $Q_{\ell +1}$, we use the definition of $n_{\ell +1}-r_{\ell +1}$ as
$$n_{\ell +1}-r_{\ell +1}=j_{\ell +1}m_{\ell +1}+m_\ell '.$$
So it suffices to prove that
\begin{equation}\label{3.31.3}Q_{\ell +1}'\subset \varphi _\ell ^{-1}(S_{m_\ell ,f_\ell }'P^0(f_\ell )),\end{equation}
where as before $S_{m_\ell ,f_\ell }'=S_{m_\ell ,f_\ell }\circ f_\ell ^{m_\ell -m_\ell '}$ 
with $S_{m_\ell ,f_\ell }'P^0$ replaced by a union of sets in the exceptional cases. 
By \ref{3.12}, in particular, (\ref{3.12.1}), $f_\ell ^{m_\ell -1}$ is a  homeomorphism on $P^{n_\ell -r_\ell }(f_\ell )$, with $P^{n_\ell -r_\ell }$ replaced by $P^{n_\ell -r_\ell ,1}(f_\ell)$ in the exceptional cases. Hence $s^{m_\ell -1}$ is a homeomorphism on $Q_{\ell +1}'$, as this is obtained by adding in complementary components to $\varphi _\ell ^{-1}(P^{n_\ell -r_\ell }(f_\ell ))$ (or $\varphi _\ell ^{-1}(P^{n_\ell -r_\ell ,1}(f_\ell ))$) on which the boundary is mapped homeomorphically. So then by definition of $m_\ell '$, in the nonexceptional cases. since $m_\ell -m_\ell '\ge r_1$ and $s^{m_\ell '}(P^{n_\ell -r_\ell })\subset \varphi _\ell ^{-1}(P^{r_1}(f_\ell ))$, we see that
$s^{m_\ell '}(Q_{\ell+1} ')\subset \varphi _\ell ^{-1}(P^0(f_\ell ))$ as required for (\ref{3.31.3}) and for the second inclusion of (\ref{3.31.2})

It remains to consider the exceptional cases and the second inclusion of (\ref{3.31.2}).  We already chose $P^{n_3-r_3,0}$ (see \ref{3.12})  so that
$$Q_3'\subset \varphi _3^{-1}(P^{n_3-r_3,0}(f_3))$$
and then 
$$Q_3\subset \varphi _3^{-1}(P^{n_3-r_3,1}(f_3)).$$
 Then recalling the definition of $U_{i,h}$ in \ref{3.12}, we see that, if $Q_2'=s^{(j_2-1)m_2}(Q_2)$ then
$$Q_i'=V_{i-1}Q_2',$$
$$Q_i=T_{m_i}^{j_i}V_{i-1}Q_2'$$
for $i\ge 4$ where  and
$$\varphi _i\circ V_{i-1}=U_{i-1,f_i}\circ \varphi _i.$$
Since $U_{3,h}=S_{m_2,h}^{j_2-1}$ this actually gives
 $$Q_4'=Q_2,$$
$$Q_4=T_{m_4}^{j_4}Q_2.$$
 Then, for $i\ge 4$,
  $$Q_i'\subset \varphi _{i}^{-1}(U_{i-1,f_{i}}P^{n_3-r_3,0}(f_{i})$$
 and hence
 $$Q_i\subset \varphi _i^{-1}S_{m_i,f_i}^{j_i}U_{i-1,1}P^{n_3-r_3,0}(f_i))=\varphi _i^{-1}(P^{n_i-r_i,1}(f_i)).$$
 
  \end{proof}
   
  \subsection{Proof of Theorem \ref{3.5}}
   Now all the machinery is in place. Assume inductively that $\ell <N$ and 
   $$f_\ell \simeq s_{y_\ell }\Amalg s_q$$ and consequently we have $\varphi _\ell $ with 
   $$\varphi _\ell \circ (s_{y_\ell }\Amalg s_q)=f_\ell \circ \varphi _\ell .$$
   We have 
   $$f_{\ell +1}, f_\ell \in V(P^0,\cdots P^{n_\ell },g).$$
  We proceed as in \ref{3.5} and \ref{3.6} in the case $\ell =1$. For all $h\in V(P^0,\cdots P^{n_\ell ,1},g)$ we have
  $$S_{m_{\ell +1},h}(P^{n_\ell ,1}(h))\subset P^{n_\ell,1 }(h)$$
  and $S_{m_{\ell +1},h}(P^{n_\ell ,1}(h))$ contains one or two points  of period $m_{\ell +1}$: exactly two unless there is a single point which is a parabolic point of $h$. Putting $h=f_\ell $ for at least one of these periodic points, which we call $z_{\ell +1}(f_\ell )$, is such that the points in $\varphi _\ell ^{-1}(z_{\ell +1}(f_\ell ))$ are of period $m_{\ell +1}$. The proof is exactly the same as in \ref{3.6}. $T_{m_{\ell +1}}$ is a  two-valued contraction on $\varphi _\ell ^{-1}(P^{n_\ell -r_\ell ,1}(f_\ell ))$, which is mapped into $\varphi _\ell ^{-1}(P^{n_\ell ,1}(f_\ell ))$, and hence also at most two-valued on $Q_\ell $. So there are at most two fixed points of $T_{m_{\ell +1}}$ which must map to at least one of the points $z_{\ell +1,i}$ since the sets $f_\ell ^i(S_{m_{\ell +1},f_\ell }(P^{n_\ell ,1}(f_\ell )))$ are disjoint for $0\le i<m_{\ell +1}$  Fix one of these fixed points of $T_{m_{\ell +1}}$ and call it $e^{2\pi iy_{\ell +1}}$. Choose $w_{\ell +1}$ such that $e^{2\pi iw_{\ell +1}}$ is in the boundary of the minor gap of $L_{y_|ell }$ and in the same component $Q_{\ell ,1}$ of $Q_\ell $ as $e^{2\pi iy_{\ell +1}}$, and of period $t_{\ell +1}m_\ell $ where $t_{\ell +1}$ is the least integer such that $S_{m_\ell}^{t_{\ell +1}}Q_{\ell ,1}\cap Q_{\ell ,1}\ne \emptyset $. This is possible by  the property (\ref{3.12.5}). There is then a tuning $f_{\ell +1/2}$ of $f_\ell $ such that
  $$f_{\ell +1/2}\simeq s_{w_{\ell +1}}\Amalg s_q$$
  and
  $$\varphi _{\ell +1/2}\circ (s_{w_{\ell +1}}\Amalg s_q)=f_{\ell +1/2}$$ and
  $$\varphi _{\ell +1/2}\circ \xi _{\ell +1/2}=\varphi _\ell $$
  where $\xi _{\ell +1/2}$ is continuous and approximated by homeomorphisms and maps the Julia set of $f_\ell $ to the Julia set of $f_{\ell +1/2}$, maps $G_n(f_\ell)$ to $G_n(f_{\ell +1/2})$  for all $n$, and maps $v_2(f_\ell )$ to $v_2(f){\ell +1/2})$ and $c_2(f_\ell )$ to $c_2(f_{\ell +1/2})$, and for $z$ in the Julia set of $f_\ell $, 
  $$\xi _{\ell +1/2}\circ f_\ell (z)=f_{\ell +1/2}\circ \xi _{\ell +1/2}(z).$$
 
   Then we can construct paths $\beta _\ell \subset P^{n_\ell ,1}(f_\ell )$ and $\zeta _\ell $ exactly as in \ref{3.7} so that $\beta _\ell $ is the union of a path in the Fatou component of $v_2(f_\ell )$ from $v_2(f_\ell )$ to $\varphi  _{\ell +1/2}(e^{2\pi iw_{\ell +1}})$. 
   $$f_\ell \circ \zeta _\ell =\beta _\ell ,$$
   $$S_{m_\ell, f_\ell}(\beta _\ell )\subset \beta _\ell ,$$
   $$f_{\ell +1}\simeq s_{\zeta _\ell }^{-1}\circ \sigma _{\beta _\ell }\circ f_{\ell +1/2}.$$
   But we also have 
   $$s_{y_{\ell +1}}\Amalg s_q\simeq \sigma _{\omega _\ell }^{-1}\circ \sigma _{\alpha _\ell }\circ s_{w_{\ell +1}}\Amalg s_q$$
   where $\alpha _\ell $ is an injective  path in $Q_{\ell }\subset S^1$ from $e^{2\pi iw_{\ell +1}}$ to $e^{2\pi iy_{\ell +1}}$ and $\omega _{\ell }$ is the path in $s^{-1}(\alpha _\ell )$ with periodic endpoints with $s\circ \eta _\ell =\alpha _\ell $. This path does exist, since $z_{\ell +1}=\varphi _\ell (e^{2\pi iy_{\ell +1}})\in S_{m_{\ell +1},f_{\ell +1/2}}(P^{n_\ell ,1}(f_\ell ))$ and hence 
   $$f^{m_{\ell +1}-1}(z_{\ell +1}(\ell +1/2))\in f_{\ell +1/2}^{-1}(P^{n_\ell ,1}(f_{\ell +1/2} )).$$
   So the proof is completed by showing that $\varphi _{\ell +1/2}(\alpha _\ell )$ is homotopic to $\beta _\ell $  via a homotopy fixing endpoints and the periodic orbit of $z_{\ell +1}(f_\ell )$, since the same will then be true for $\varphi _{\ell +1/2}(\eta _\ell )$ and $\zeta _\ell $. After homotopy we can assume that $\beta _\ell $ is the union of a path in the Fatou component of $v_2(f_\ell )$ to the point $\varphi _\ell (e^{2\pi iw_\ell })$ in the boundary, and a path $\beta _\ell '$ from $\varphi _{\ell +1/2}(e^{2\pi iw_{\ell +1}})$ to $z_{\ell +1}(f_{\ell +1/2})$ such that 
   $$\beta _\ell '\cap S_{m_\ell }^i(\beta _\ell ')=\emptyset ,\mbox{ for }0<i<t_{\ell +1}$$
    and 
   $$S_{m_\ell t_{\ell +1},f_{\ell +1/2}}\beta _\ell \subset \beta _\ell ',$$
   $$\beta _\ell '\subset f^{m_\ell }(\beta _\ell ')\mbox{ or }\beta _\ell '\cap f^{m_\ell }(\beta _\ell ')=\emptyset ,\mbox{ depending on whether }t_{\ell +1}=1\mbox{ or }t_{\ell +1}>1.$$
   We have similar properties for $\varphi _{\ell +1/2}(\alpha _\ell ')$ since  $\alpha _\ell $ is the union of a path in the minor gap of $L_{y_\ell }$ and a path $\alpha _\ell '\subset Q_{\ell ,1}\subset S^1$ from $e^{2\pi iw_{\ell +1}}$ to $e^{2\pi iy_{\ell +1}}$ which  is  contained in or disjoint from $s^{m_\ell }(\alpha _\ell ')$ and  $T_{t_{\ell +1}m_{\ell }}(\alpha _\ell '\subset \alpha _\ell '$, depending on whether $t_{\ell +1}=1$ or $t_{\ell +1}>1$,where $T_{t_{\ell +1} m_\ell}$ is the branch of $T_{m_\ell }^{t_{\ell +1}}$ with $\varphi _\ell \circ T_{t_{\ell +1}m_\ell }=S_{t_{\ell +1}m_\ell ,f_\ell }\circ \varphi _\ell $ (which is the same as the corresponding statement being true with $\ell +1/2$ replacing $\ell $). Extend $\alpha _\ell '$ beyond $e^{2\pi iy_\ell }$ to the boundary of $Q_{\ell ,1}$ to a path $\alpha _{\ell ''}$ which still satisfies $T_{t_{\ell +1}m_\ell}(\alpha _\ell '')\subset \alpha _\ell ''$.
   
   Then as in \ref{3.9}, using invariance, we can show that $\beta _\ell $ and $\varphi _{\ell +1/2} (\alpha _\ell )$ are homotopic via a homotopy preserving endpoints and the forward orbit of $z_{\ell +1}(\ell +1/2)$. By \ref{3.12}, for $\ell \ge 2$, the only point in the forward orbit of $z_{\ell +1}(\ell +1/2)$ which might prevent the homotopy is $f_{\ell +1/2}^{m_\ell}(z_{\ell +1}(\ell +1/2))$. If the homotopy does not hold, then we look at the last essential intersection of $\beta _\ell $ with $\varphi _{\ell +1/2}(\alpha _\ell ')$, map forward under $f_\ell ^{m_\ell }$  and obtain a contradiction.

   \end{document}